\documentclass[12pt]{amsart}
\usepackage[cp1250]{inputenc}
\usepackage[T1]{fontenc}
\usepackage{amscd}

 \DeclareMathOperator{\dist}{dist}

 \DeclareMathOperator{\pr}{pr}

\theoremstyle{plain}
\newtheorem{thm}{Theorem}[section]
\newtheorem{lem}[thm]{Lemma}
\newtheorem{prop}[thm]{Proposition}
\newtheorem{cor}[thm]{Corollary}
\newtheorem{con}[thm]{Conjecture}

\theoremstyle{definition}
\newtheorem{dff}[thm]{Definition}
\newtheorem{rem}[thm]{Remark}

\newtheorem{conv}[thm]{Convention}

\numberwithin{equation}{section}
\def\a{\alpha}

\def\bx{\bar x}
\def\by{\bar y}
\def\cpp{C_{\phi,\psi}}

\def\C{\mathcal{W}}

\def\ab{A^{\beta}}

\def\H{\mathcal{H}}

\def\I{\mathcal{I}}
\def\U{\mathcal{U}}
\def\V{\mathcal{V}}

\def\L{\mathcal{L}}
\def\X{\frak{X}}
\def\l{*}

\def\g{\Gamma}

\def\lff{\lambda_{\f}}

\def\phi{\varphi}
\def\p{\partial}
\def\pp{\phi}
\def\fp{f^{\psi}}
\def\cc{\CC^{\infty}}
\def\as{\a_{st}}

\def\ci{\circ}

\def\cfr{C_{\phi,r}}

\def\rz{\mathbb{R}}
\def\R{\rz}

\def\N{\mathbb{N}}

\def\wyr#1{\textit{#1}}
\def\Z{\mathbb{Z}}
\def\s{\subset}
\def\ss{a_n}
\def\cm{\CC^{\infty}(M)}
\def\t{\times}
\def\r{\rightarrow}
\def\cc{\CC^{\infty}}
\def\dnk{\cc(\R^m,\R^m)}
\def\as{\alpha_{st}}
\def\ac{\as}

\def\rr{\R^{m}}
\def\crr{\CC^{\infty}(\R^{m})}
\def\dz{{\partial\over \partial x_0}}
\def\x0{x_0}
\def\dxi{{\partial\over\partial x_i}}
\def\dx0{{\partial\over\partial x_0}}
\def\dyi{{\partial\over\partial y_{i}}}
\def\ld{,\ldots,}
\def\pp{\partial}
\def\f{{\bf f}}
\def\gd{\Gamma^{\dd}}

 \DeclareMathOperator{\cont}{Cont}
 \def\con{\cont_c(M,\alpha)_0}
 \def\cork{\cont_c(\C_k^m,\as)_0}
 \def\corr{\cont_c(\R^{m},\as)_0}
 \DeclareMathOperator{\diff}{Diff}
 \DeclareMathOperator{\id}{id}
\DeclareMathOperator{\fl}{Fl}
 \DeclareMathOperator{\supp}{supp}
\DeclareMathOperator{\dd}{d} \DeclareMathOperator{\dom}{dom}
\DeclareMathOperator{\CC}{C}

\thanks{Supported in part by the AGH grant n. 11.420.04}
\address{Faculty of Applied Mathematics, AGH University of Science and
\linebreak Technology, al. Mickiewicza 30, 30-059 Krak\'ow,
Poland} \email{tomasz@uci.agh.edu.pl}
\date{March 3, 2010;  revised version}

\title{Commutators of contactomorphisms }
\author{ Tomasz Rybicki}

\begin{document}

\maketitle

\begin{abstract}
The group of volume preserving diffeomorphisms, the group of
symplectomorphisms and the group of contactomorphisms constitute
the classical groups of diffeomorphisms. The first homology groups
of the compactly supported identity components of the  first two
groups have been computed by Thurston and Banyaga, respectively.
In this paper we  solve  the long standing problem  on the
algebraic structure of the third classical diffeomorphism group,
i.e. the contactomorphism group. Namely we show that  the
compactly supported identity component of the group of
contactomorphisms is perfect and simple (if the underlying
manifold is connected). The result could be applied in various
ways.
\end{abstract}

\section{Introduction}

Let $(M,\alpha)$ be a contact manifold, i.e. $M$ is a $C^\infty$
smooth paracompact manifold of dimension $m=2n+1$, $m\geq 3$, and
$\alpha$ is a $C^\infty$ $1$-form on $M$ such that
$\alpha\wedge(d\alpha)^n$ is a volume form. A contactomorphism $f$
of $(M,\alpha)$ is a $C^\infty$ diffeomorphism of $M$ such that
$f^*\alpha=\lambda_f\alpha$, where $\lambda_f$ is a smooth nowhere
vanishing function on $M$ depending on $f$. In other words, a
contactomorphism $f$ is a diffeomorphism whose tangent map $Tf$
preserves the $C^\infty$ contact hyperplane field $\H=\ker\alpha$.
Notice that contactomorphisms of $(M,\alpha)$ are determined by
the  contact hyperplane field $\H$.

Let $\cont(M,\alpha)$ denote the group of contactomorphisms of
$(M,\alpha)$, and let $\cont_c(M,\alpha)$ be its compactly
supported subgroup. Observe that $\cont(M,\alpha)$ carries  the
structure of an infinite dimensional Lie group (see, e.g., [9]).
Then, in view of the local contractibility of $\cont_c(M,\alpha)$,
its identity component $\con$ coincides with all $f\in\cont(M,\a)$
which can be joined with the identity by a smooth isotopy in
$\cont_c(M,\a)$.
 Our main result is the following

\begin{thm}
The group $\con$ is perfect, that is $\con$ is equal to its own
commutator subgroup.
\end{thm}
Epstein in [4] proved that the commutator subgroup of a group of
homeomorphisms satisfying some natural conditions is simple. It is
easily checked that Epstein's conditions are satisfied by $\con$
(section 10). Therefore we have

\begin{cor}
If $M$ is connected then the group $\con$ is simple.
\end{cor}

The contactomorphism group is a classical group of
diffeomorphisms. Since the well-known results of Herman [8],
Thurston [17] and Mather [12] on the simplicity of
$\diff^r_c(M)_0$, $r=1\ld \infty$, $r\neq \dim(M)+1$,  the problem
of the perfectness (or of computing the first homology group) of
groups of diffeomorphisms have been studied in several papers.
First of all such studies have been done on the classical groups
of diffeomorphisms.

An essential feature of the geometry and topology of manifolds
with a volume or a symplectic form is the existence  of
invariants, called the flux homomorphisms. According to the
celebrated results by Thurston [16] and Banyaga [1] (see also [2])
the first homology groups $H_1(\diff_c(M,\omega)_0)$ , where
$\omega$ is a volume or a symplectic form, can be expressed by
means of the flux homomorphism and other invariants, and they
depend also on the compactness of the underlying manifold. Notice
that the results and the methods of their proofs in both cases are
similar. In general, the compactly supported identity components
of the volume preserving diffeomorphism group and the
symplectomorphism group are not perfect. Note that Banyaga's
theorem was generalized to the locally conformal symplectic
structures in [7].

 A basic reason that $\con$ is perfect is the fact that in
the contact case there do not exist invariants analogous to the
flux homomorphism and, consequently, a fragmentation property
holds in its usual form. In view of this fact Theorem 1.1 was
conjectured, e.g. in [2]. A main obstacle to find a proof similar
to that of Thurston [16] is the lack of a canonical contact
structure on the torus $T^m$, a clue ingredient of a hypothetical
proof by this method. Canonical contact structures do exist,
however, on the cylinders $\C^m_k=(\mathbb S^1)^k\t\R^{m-k}$,
$k=1\ld n+1$, and this fact is essential in our proof.

The fragmentation property (Lemma 5.2) is, in fact, an
indispensable ingredient of the proof. Nevertheless, it is
probably not a sufficient tool to prove Theorem 1.1. My idea is to
use in the proof also a fragmentation of contactomorphisms in a
neighborhood of the identity of $\corr$ (section 5). I call it a
fragmentation of the second kind. An essential advantage of such
fragmentations is that the factors of the resulting decomposition
are uniquely determined by the initial contactomorphism. Moreover,
the norms of these factors are controlled by the norm of the
initial contactomorphism in a convenient way.

The proof  consists in an application of Schauder-Tichonoff's
fixed point theorem to some operator in a functional space. The
origins of this method were explained in Epstein  [5], where it
was used  to give an alternative proof of the perfectness of
$\diff^{\infty}_c(M)_0$. We would like to stress, however,  that
 several parts of the
proof for diffeomorphisms cannot be carried over  to the contact
case and some new ideas and technical refinements in the proof of
Theorem 1.1 are indispensable. Our construction of a fixed point
operator consists
 of ten steps (c.f. section 9) and functional spaces on various domains must be considered in it.

A crucial  step in the proof is the use of a rolling-up operator
$\Psi_A$ defined in section 8 (Proposition 8.7). Such  operators
are used in [12] and [5], and  analogous operators exist for the
group $\corr$, but only with respect to the first $n+1$ variables.
However, they are useless since the property
\[\forall f\in\dom(\Psi_A),\quad
[\Psi_A(f)]=[f]\quad\hbox{in}\quad H_1(\corr),\]
 a clue part of the proof in [12], does not hold in the
contact category for very basic reasons. In this situation we
construct a new rolling-up operator $\Psi_A$ for $\corr$ by means
of auxiliary operators acting on contactomorphisms on the
subsequent contact cylinders $\C^m_k$, $k=1\ld n$. An essential
fact is that a "remainder" contactomorphism living on the last
cylinder $\C^m_{n+1}$ possesses a representant in the commutator
subgroup of $\corr$ (Lemma 8.6). This ensures the above property
for $\Psi_A$. Such an argument is no longer true for
$\diff^r_c(\R^m)_0$ and, consequently, the proof of Theorem 1.1
cannot be carried over to the case of diffeomorphisms.

The contact topology and geometry are intensively studied
nowadays, c.f. [6]. Theorem 1.1, which is a contact analog of the
theorems of Thurston and Banyaga, could be possibly applicable in
various ways. In the last section we indicate two directions of
such applications. The most important seems to be the fact that
due to Theorem 1.1 the commutator length is a conjugation
invariant norm on $\con$.  For the significance of Banyaga's
theorem in  the symplectic topology, see, e.g., [14] and [3].

In the appendix it is  observed that the universal covering group
of $\con$ is also perfect by an  argument similar to that for
$\con$.

\bigskip\wyr{Acknowledgments}. I would like to thank very much the
referee for pointing out some mistakes and misprints. These
remarks helped me to correct  the paper.

\section{The group of contactomorphisms}

Let $M$ be a smooth manifold with $\dim(M)=m=2n+1$ and let
$\alpha$ be a contact form on $M$. A contact form $\alpha$ can be
put into the following normal form. For any $p\in M$ there is a
chart $(x_0,x_1\ld x_n,y_1\ld y_n):M\supset U\r u(U)\subset\rr$,
centered at $p$, such that $\a|_U=\dd x_0-y_{1}\dd x_1-\ldots
-y_{n}\dd x_n$.

 The symbol $ \X(M,\a)$ will stand for
the Lie algebra of all contact vector fields, i.e. $X\in \X(M,\a)$
iff $L_X\a =\mu_X \a $ for some function $\mu_X\in\cm$, where $L$
is the Lie derivative. Let $ \X_c(M,\a)$ be the Lie subalgebra of
compactly supported elements of $ \X(M,\a)$.

Let $h\in \con$ and $\{ h_t\}_{t\in I}$ be a smooth isotopy such
that $h_1=h, h_0=\id$ and each $h_t$ stabilizes outside a fixed
compact $K\subset M$. Of course, such a smooth contact isotopy
determines a smooth family of contact vector fields
$X_t\in\X_c(M,\a)$, namely for $p\in M$ we have
\begin{equation}
         {\p h_t\over \p t}(p)=X_t(h_t(p)).
\end{equation}
In fact, one has $L_{X_t}\a =\mu _{X_t}\a $ with $\mu _{X_t} =(\p
\ln\lambda _{h_t}/\p t)h_t^{-1}$ where $h_t^*\a =\lambda_{h_t} \a
$.

Let $X_{\a}$ denote the unique vector field satisfying
$i_{X_{\a}}\a =1$ and $i_{X_{\a}}\dd\a = 0$. $X_{\a}$ is called
the Reeb vector field. A  vector field $X$ is called \wyr
{horizontal} if $i_X\a =0$. A dual concept is a \wyr {semibasic}
form, i.e. any 1-form $\gamma $ such that $\gamma (X_{\a})=0,$ and
the duality is established by the isomorphism $d\a :X\longmapsto
i_X\dd\a $. It follows the isomorphism of vector bundles
\begin{equation}
I_{\a}:TM\ni X\mapsto i_X\dd\a+\a(X)\a\in T^*M.
\end{equation}

As a consequence we have the existence of the isomorphism
$\I_{\a}$ below (c.f. Libermann [10]), an important tool in the
contact geometry.

 \begin{prop}  There is an isomorphism $\I_{ \a} :
\X(M,\a) \rightarrow \CC^{\infty }(M)$ by $\I_{\a}(X)=i_X\a $. For
$H\in \CC^{\infty }(M)$ we have \[\I_{ \a}^{-1}(H)=HX_{\a}+(d\a
)^{-1}((i_{X_{\a}}\dd H)\a -\dd H).\]
\end{prop}

We will deal with  the standard contact form $\as=\dd
\x0-\sum_{i=1}^n y_{i}\dd x_i$ on $\rr$. Then we have
$X_{\as}=\dz$ and $\dd\as=\sum_{i=1}^n\dd x_i\wedge \dd y_{i}$.

Notice that the isomorphism $I_{\as}:T\R^m\r T^*\R^m$ is
independent of the variables $x_i$, $i=0\ld n$. Likewise, the
isomorphism $\I_{\as}:\frak X(\R^m,\as)\r\CC^{\infty}(\R^m)$ sends
vector fields independent of $x_i$ to functions independent of
$x_i$ and vice versa.

Observe that $\H=\ker(\as)$ is generated by $Y_i=\dyi$ and
$X_i=\dxi+y_{i}\dz$, where $i=1\ld n$.

 Next it is easily seen that $d\as(Y_i)=-\dd
x_i$ and $d\as(X_i)=\dd y_{i}$. Every contact vector field
$X=u_0\dz+\sum_{i=1}^n u_i\dxi+u_{n+i}\dyi\in\X(\rr,\as)$ is
identified by $\I_{\as}$ with the function
$H=i_X\as=u_0-\sum_{i=1}^n y_{i}u_i\in\crr$. Conversely, in view
of Proposition 2.1 and the above equalities, we have
\begin{equation}
 X_H=\left(H-\sum_{i=1}^n y_{i}{\pp H\over \pp y_{i}}\right)\dz -\sum_{i=1}^n
   {\partial H\over \partial y_{i}}\dxi
  +\sum_{i=1}^n\left({\partial H\over \partial x_i}+y_{i}{\partial H\over \partial
  x_0}\right)\dyi,
\end{equation}
where $X_H=\I_{\as}^{-1}(H)$ for all $H\in\crr$. Now we wish to
specify some elements in $\corr$. The following contact vector
fields on $\rr$ and their flows will be of use. Throughout we will
often write $x$ instead of $(x_1\ld x_n)$ and $y$ instead of
$(y_{1}\ld y_{n})$.

\begin{enumerate}
  \item Let $H_0$ be the constant function 1. Then $X_{H_0}=\dx0=X_{\as}$ and its flow takes the
  translation form
  $\fl^{H_0}_t(x_0,x,y)=(x_0+t,x,y)$.
  \item Put $H_i(\x0,x,y)=-y_{i}$ ($i=1\ld n$). Then $X_{H_i}=\dxi$ and its
  flow consists of the translations
  $\fl^{H_i}_t(\x0,x,y)=(\x0,x+t{\bf 1}_i,y)$.
  \item  Let
  $H_{n+i}(\x0,x,y)=x_i$ ($i=1\ld n$). From (2.3) we obtain
  $  X_{H_{n+i}}=\dyi+x_i\dz.$
  Hence $\fl^{H_{n+i}}_t(\x0,x,y)=(\x0+tx_i,x,y+t{\bf 1}_i)$.
  \item For $H(\x0,x,y)=2\x0-\sum_{i=1}^n x_iy_{i}$, we have
  $X_H=2\x0\dz+\sum_{i=1}^n(x_i\dxi+y_{i}\dyi)$, and its flow assumes the form
$\fl^{H}_t(\x0,x,y)=(e^{2t}\x0,e^tx,e^ty)$. By
$\chi_{a}:\rr\r\rr$, where $\chi_{a}(\x0,x,y)=(a^2\x0,a x,a y)$,
i.e. $\chi_a=\fl^H_{\ln a}$, we will denote the resulting contact
homothety. \item Let $\bar H(x,y,z)=\x0-\sum_{i=1}^n x_iy_i$. Then
  $X_{\bar H}=\sum_{i=0}^n x_i\dxi$, and its flow satisfies
$\fl^{\bar H}_t(\x0,x,y)=(e^t\x0,e^tx,y)$. Let
$\eta_{a}:\rr\r\rr$, where $\eta_{a}(\x0,x,y)=(a\x0,a x, y)$, i.e.
$\eta_a=\fl^{\bar H}_{\ln a}$, denote   the resulting map.

\end{enumerate}

Denote $\tau_{i,t}=\fl^{H_i}_t$, $i=0,1\ld n$, and
$\sigma_{i,t}=\fl^{H_{n+i}}_t$, $i=1\ld n$.   The supports of
$\tau_{i,t}$, $\sigma_{i,t}$, $\chi_a$ and $\eta_a$ are not
compact . But if we take the product of $H_i$, $H$, or $\bar H$
with a suitable bump function we will obtain elements of $\corr$
which are equal to the previous contactomorphisms on a
sufficiently large interval.Abusing the notation we will denote
all these elements of $\corr$ by the same letters as before. This
ambiguity will not matter in the proof  and we will not mention it
in the sequel.

Observe that the translations along the $y_i$ axes are not
contactomorphisms since they do not preserve the contact
distribution.

\begin{prop}\begin{enumerate}
\item A diffeomorphism $f$ of $\R^m$ is a contactomorphism if and
only if for $f=(f_0,f_1\ld f_{2n})$ we have
\begin{align*}
{\p f_0\over \p x_0}-\sum_{j=1}^nf_{n+j}{\p f_j\over \p
x_0}&=\lambda_f,\\
{\p f_0\over \p x_i}-\sum_{j=1}^nf_{n+j}{\p f_j\over \p
x_i}&=-y_i\lambda_f,\quad i=1\ld n,\\
{\p f_0\over \p y_i}-\sum_{j=1}^nf_{n+j}{\p f_j\over \p
y_i}&=0,\quad i=1\ld n.
\end{align*}
\item If $f\in\cont(\R^m,\as)$ is independent of $x_0$ (i.e. ${\p
(f-\id)\over\p x_0}=0$) then $\lambda_f=1$. If, in addition, $f$
is independent of $x_i$, $i=1\ld n$, then $f_{n+j}(x_0,x,y)=y_j$
for $j=1\ld n$.
\end{enumerate}
 \end{prop}

 \begin{proof}
(1) This is the equality $f^*\as=\lambda_f\as$ written in
coordinates.

(2) It follows immediately from the first $n+1$ equalities.
 \end{proof}

A crucial idea in the proof of Theorem 1.1 is to consider groups
of contactomorphisms on cylinders which admit standard contact
structures. Let us denote
\begin{equation}
 \C^m_k:=(\mathbb
S^1)^k\t\R^{m-k},\end{equation} and for a constant $A>0$
\begin{equation} E_A^{(k)}:=(\mathbb S^1)^k\t[-A,A]^{m-k},\quad
k=1\ld n+1.
\end{equation}
It will be often convenient to write
\begin{equation}
\C^m_0:=\R^m,\quad E^{(0)}_A=E_A:=[-A,A]^m.
\end{equation}
 The coordinates of $\C^m_k$ will be written
$(\xi_0,\xi_1\ld\xi_n,y_1\ld y_n)$, where $\xi_i$ is the natural
coordinate on $\mathbb S^1$ for $i=0\ld k-1$, and $\xi_i=x_i$ for
$i=k\ld n$. For short, we will often write $\xi$ instead of
$(\xi_1\ld \xi_n)$, that is the natural coordinates on $\C^m_k$
will be denoted by $(\xi_0,\xi,y)$. On the cylinder $\C^m_k$ we
have the standard contact form given by $\as=\dd \xi_0-y_{1}\dd
\xi_1-\ldots -y_{n}\dd \xi_n$.

If $E\s\C^m_k$ the symbol $\cont_E(\C^m_k,\as)$ stand for the
totality of elements of  $\cont_c(\C_k^{m},\as)_0$ with support
included in $E$.
 The description of a chart in $\cork$ at the
identity will be given in section 4.

Observe that the translations $\tau_{i,t}$ still live on $\C^m_k$
whenever $k\leq i\leq n$.

\section{Basic estimates}

Let  $r$ be a nonnegative integer. For $g\in \cc(\rz^m,\rz^{m'})$
we define
\begin{equation*}
  \|D^rg\|=\sup_{p\in\rz^m}|D^rg(p)|=\sup_{p\in\rz^m}\sup_{|u_1|\leq 1\ld|u_r|\leq 1}
  |D^rg(p)(u_1\ld u_r)|\leq\infty,
\end{equation*}
where $D^0g=g$. Next, for a diffeomorphism
$f\in\diff^{\infty}(\R^m)$ we put for any $r\geq 0$
\[\mu_r(f)=\|D^r(f-\id)\|,\]

\begin{equation*}
  M_r(f)=\max\{\mu_0(f),\mu_1(f),\ldots,\mu_r(f)\}.
\end{equation*}
If $\f=(f_1\ld f_k)$ then we define \[\mu_r(\f)=\sup_{i=1\ld
k}\mu_r(f_i),\quad M_r(\f)=\sup_{i=1\ld k}M_r(f_i).\]

We have $\mu_1(f)\leq\|Df\|+1$, $\|Df\|\leq\mu_1(f)+1$ and
$\mu_r(f)=\|D^rf\|$ for $r\geq2$. Let $E\s\R^m$ be a closed set.
We define
\begin{equation}
R_E=\sup_{p\in E}\dist(p,\overline{\R^m\setminus E})\leq\infty.
\end{equation}
\begin{prop} Let $R_E<\infty$.
If $f$ is a diffeomorphism and $\supp(f)\s E$,    then
\begin{equation*}
  \mu_{r}(f)\leq C\mu_{r+1}(f),
  \end{equation*}
where $r\geq 0$ and the constant $C$ depends on $R_E$.
\end{prop}

In fact, the  inequality is obtained by integrating  partial
derivatives of the map $f-\id$.

Let $f,g\in\dnk $ and $r\geq 1$. Then we have
\begin{equation}
  D(f\circ g)=  (Df\circ g) Dg,
\end{equation}

\begin{equation}
\begin{split}
  D^r(f\circ g)&=  (D^rf\circ g)(Dg\times\ldots\times Dg)
  +(Df\circ g) D^r g\\
  & + \sum C_{i;j_1,\ldots ,j_i}(D^i f\circ g)
  (D^{j_1}g\times\ldots\times D^{j_i}g).
\end{split}
\end{equation}
It follows from (3.2) and (3.3) the equalities
\begin{equation}
D(f^{-1})=(Df)^{-1}\circ f^{-1},
\end{equation}
\begin{equation}\label{eq:derFormInv}
\begin{aligned}
  D^r(f^{-1})&=D(f^{-1})(D^{r}f\circ f^{-1})(D(f^{-1})\times
  \ldots\times D(f^{-1})) \\
  &+D(f^{-1})\sum C_{i;j_1,\ldots,j_i}(D^{i}f\circ
  f^{-1})(D^{j_1}(f^{-1})\times\ldots\times
  D^{j_i}(f^{-1})).
\end{aligned}
\end{equation}

 In (3.3) and (3.5) the sum is taken over $1<i<r$, $1\leq j_s$, $j_1+\ldots+j_i=r$
and $C_{i;j_1,\ldots,j_r}$ are positive integers. Note that in
each term of the above sum there exists $j_s>1$.

\begin{dff}By polynomials we will understand polynomials with nonnegative
coefficients.

An \wyr{admissible} polynomial is a polynomial  without constant
and linear terms. Admissible polynomials will be denoted  by $F$
with some indices. We will also consider polynomials  without
constant term. Such polynomials will be designated by $P$ with
some indices.
\end{dff}
\begin{conv}
In order to avoid repeating that either  polynomials,  or
constants  depend on some values, we adopt the following
convention. If, e.g., a polynomial $P$  depends on $\psi$, $r$,
and $A$, then we will write $P_{\psi,r,A}$, i.e. all the values
determining a given object will appear as subscripts. The only
exception is that  we will not mention explicitly the dependence
on $m=\dim(M)$.

 In the sequel we will often omit the sign of composition
$\ci$.

\end{conv}

By using (3.2)-(3.5) and the induction argument we have the
following lemma (c.f. [12]).
\begin{lem}\label{lem:estimations}
\begin{enumerate}
  \item For any $f_1,\ldots,f_k\in\dnk$ and $\f=(f_1\ld f_k)$
  \begin{equation*}
    \mu_1(f_1\circ\ldots\circ f_k)\leq k\mu_1(\f)(1+\mu_1(\f))^{k-1}.
  \end{equation*}
  \item For $r,k\geq2$ there exists an
  admissible polynomial $F_{r,k}$  such that for any $f_1,\ldots,f_k\in\dnk$, $\f=(f_1\ld
  f_k)$,
  one has
  \begin{equation*}
    \mu_r(f_1\circ\ldots\circ f_k)\leq
    k\mu_r(\f)(1+\mu_1(\f))^{r(k-1)}+F_{r,k}(M_{r-1}(\f)).
  \end{equation*}
  \item If $f\in\diff(\rz^m)$ with $\mu_1(f)<1$,  then
 \begin{equation*}
    \mu_1(f^{-1})\leq \frac{\mu_1(f)}{1-\mu_1(f)}.
  \end{equation*}
  \item   For any $r\geq2$ there exists
  an admissible polynomial $F_r$
  such that for any   $f\in\diff(\rz^m)$ with $\mu_1(f)<{1\over 2}$
\begin{equation*}
    \mu_r(f^{-1})\leq\mu_r(f)(1+2\mu_1(f))^{r+1}+F_r(M_{r-1}(f)).
  \end{equation*}
\end{enumerate}
\end{lem}

In the group $\corr$ we will need  more specified norms. For any
$f\in \cont_c(\rr,\as)$ and $r\geq 0$ we put
\[\mu^{\l}_r(f)=\max\{\|D^r(f-\id)\|,\|D^r(\lambda_f-1)\|\}\]
 and
\begin{equation*}
  M^{\l}_r(f)=\max\{\mu^{\l}_0(f),\mu^{\l}_1(f),\ldots,\mu^{\l}_r(f)\}.
\end{equation*}
Here $\lambda_f\in\cc(\rr)$ such that $f^*\as=\lambda_f\as$. We
define $\mu^{\l}_r(\f)$ and $M^{\l}_r(\f)$ for $\f=(f_1\ld f_k)$
analogously as above.
\begin{prop} Let $R_E<\infty$.
If  $f\in \cont_c(\rr,\as)$, and $\supp(f)\s E$, then
\begin{equation*}
  \mu^{\l}_{r}(f)\leq C\mu^{\l}_{r+1}(f),
  \end{equation*}
 where $r\geq 0$ and $C$ depends on $R_E$.
\end{prop}

Indeed, the  inequality follows from  the definition of
$\mu_r^{\l}$ by integrating  partial derivatives of the maps
$f-\id$ and $\lambda_f-1$.

\begin{lem}\label{lem:estimationss}
\begin{enumerate}
  \item For any $f_1,\ldots,f_k\in
\cont_c(\rr,\as)$ and $\f=(f_1\ld f_k)$ we have
  \begin{equation*}
    \mu_1^{\l}(f_1\circ\ldots\circ f_k)\leq k\mu_1^{\l}(\f)\big((1+\mu^{\l}_0(\f))
    (1+\mu^{\l}_1(\f))\big)^{k-1}.
  \end{equation*}
  \item  For $r,k\geq2$ there exists an
  admissible polynomial $F_{r,k}$  such that for any $f_1,\ldots,f_k\in
\cont_c(\rr,\as)$, $\f=(f_1\ld f_k)$, one has
  \begin{equation*}
    \mu_r^{\l}(f_1\circ\ldots\circ f_k)\leq
    k\mu_r^{\l}(\f)(1+\mu_0^{\l}(\f))^{k-1}
    (1+\mu_1^{\l}(\f))^{r(k-1)}
    +F_{r,k}(M^{\l}_{r-1}(\f)).
  \end{equation*}
  \item If $f\in\cont_c(\rr,\as)$ with $\mu^{\l}_0(f)< \frac{1}{2}$ and $\mu_1^{\l}(f)<\frac{1}{2}$,  then
 \begin{equation*}
    \mu_1^{\l}(f^{-1})\leq 8\mu_1^{\l}(f).
  \end{equation*}
  \item For any  $r\geq2$ there exists an admissible polynomial $F_r$ such that for any
  $f\in\cont_c(\rr,\as)$ with $\mu^{\l}_0(f)< \frac{1}{2}$ and $\mu_1^{\l}(f)<{1\over 2}$
  one has
  \begin{equation*}
    \mu_r^{\l}(f^{-1})\leq 2^{r+2}\mu_r^{\l}(f)(1+2\mu_1^{\l}(f))^{r+1}+F_r(M_{r-1}^{\l}(f)).
  \end{equation*}
\end{enumerate}
\end{lem}
\begin{proof}
First notice that
\begin{equation}
\lambda_{f_1\ci\cdots\ci f_k}=(\lambda_{f_1}\ci f_2\ci\ldots\ci
f_k)\cdot(\lambda_{f_2}\ci f_3\ci\cdots\ci
f_k)\cdot\ldots\cdot\lambda_{f_k}
\end{equation}
and
\begin{equation}
\lambda_{f^{-1}}={1\over \lambda_{f}\ci f^{-1}}
\end{equation}
for any $f,f_1\ld f_k\in\cont_c(\rr,\as)$. For $\f=(f_1\ld f_k)$
denote $\lff=\sup_{i=1\ld k}\|\lambda_{f_i}\|$. In order to show
(1) observe that in view of (3.6) and (3.2) we have
 \begin{align*}
   & \|D\lambda_{f_1\circ\ldots\circ f_k}\|\\
&\leq \lff^{k-1} \big(\|D(\lambda_{f_1}\ci f_2\ci\cdots\ci
f_k)\|+\|D(\lambda_{f_2}\ci
f_3\ci\cdots\ci f_k)\|+\cdots +\|D\lambda_{f_k}\|\big)\\
&\leq \lff^{k-1} (\|D\lambda_{f_1}\|\|D f_2\|\cdots\|D f_k\|
+\|D\lambda_{f_2}\|
\|Df_3\|\cdots\|D f_k\|+\cdots +\|D\lambda_{f_k}\|)\\
&\leq
\lff^{k-1}\mu_1^{\l}(\f)\big((1+\mu_1^{\l}(\f))^{k-1}+(1+\mu_1^{\l}(\f))^{k-2}
+\cdots+1\big)\\
 &\leq
k\mu_1^{\l}(\f)(1+\mu_0(\f))^{k-1}
    (1+\mu_1^{\l}(\f))^{k-1}.
\end{align*}
 Here  we used
the inequalities $\|\lambda_{f_i}\|\leq 1+\mu^{\l}_0(f_i)$,
$\|D\lambda_{f_i}\|\leq\mu_1^{\l}(f_i)$ , and $\|Df_i\|\leq
1+\mu^{\l}_1(f_i)$, for $i=1\ld k$. Combining this with Lemma
3.4(1) we obtain (1). (2) follows analogously by (3.3), (3.6) and
Lemma 3.4(2).

Next, (3) follows from the trivial inequality
$\frac{\mu_1(f)}{1-\mu_1(f)}\leq\frac{\mu_1^{\l}(f)}{1-\mu_1^{\l}(f)}$
and
\begin{align*}\|D\lambda_{f^{-1}}\|&\leq\|D(1/\lambda_f)\|\|Df^{-1}\|\leq
4\|D\lambda_f\|(1+\mu_1(f^{-1}))\\
&\leq 4\mu_1^*(f)(1+2\mu_1(f))\leq 8\mu^*_1(f)\end{align*} in view
of (3.7), $\|\lambda_f\|>\frac{1}{2}$ and Lemma 3.4. Finally, in
order to show (4)  observe , in view of (3.7), (3.3) and Lemma
3.4, that
\begin{align*}
&\|D^r\lambda_{f^{-1}}\|\leq\|D^r(1/\lambda_f)\|\|D
f^{-1}\|^r+\|D(1/\lambda_f)\|\|D^r
f^{-1}\|+F^1_r(M^{\l}_{r-1}(f))\\
&\quad \leq 4\|D^r\lambda_f\|\|Df^{-1}\|^r+4\|D\lambda_f\|\|D^rf^{-1}\|+F^2_r(M^{\l}_{r-1}(f))\\
&\quad\leq
2^{r+2}\mu_r^{\l}(f)(1+2\mu_1^{\l}(f))^{r+1}+F_r(M_{r-1}^{\l}(f)),
\end{align*}
as $\|\lambda_f\|>\frac{1}{2}$. Now (4) follows from the above
inequality and Lemma 3.4(4).
\end{proof}

\begin{rem}
 Note that Lemma 3.6 remains true for contactomorphisms
on $\C^m_k$, c.f. (2.4), from a sufficiently small
$C^1$-neighborhood of the identity. The reason is that if we
estimate the norms of these elements at a point then the r.h.s. of
the inequalities in question may be written locally, that is in
$\R^m$.

By a subinterval of $E^{(k)}_A\s\C^m_k$, c.f. (2.5),  we
understand a subset of $\C^m_k$ of the form $(\mathbb S^1)^k\t
E'$, where $E'$ is a subinterval of $[-A,A]^{m-k}$. If we put
\begin{equation}
R_E=\sup_{p\in E}\dist(p,\overline{\C^m_k\setminus E}),
\end{equation}
then Proposition 3.5 still holds for $\cork$ instead of $\corr$.
\end{rem}

\section{Description of a chart}

It is well-known that $\cont(M,\a)$ admits an infinite dimensional
Lie group structure (see Lychagin [11], or the elegant proof in
Kriegl and Michor [9]). In particular, this group is locally
contractible.

Observe that for an arbitrary diffeomorphism $f$ of $M$ endowed
with a contact form $\a$ we may define $\lambda_f\in\cc(M)$ by
\begin{equation}
\lambda_f=i_{X_{\a}}\lambda_f\a=i_{X_{\a}}(f^*\a)=f^*(i_{f_*X_{\a}}\a),
\end{equation}
where $i$ designates the interior product.
 The construction of charts on the group
$\cont(M,\a)$ is based on the fact that a diffeomorphism $f$ is a
contactomorphism if and only if the graph of $(f,\lambda_f)$ ,
  \[\{(p,f(p),\lambda_f(p)):\, p\in M\},\]
  is a Legendrian submanifold of $(\tilde M,\tilde\alpha)$,
  where $\tilde M=M\t M\t\R\setminus 0$, $\tilde\alpha=t\pr_1^*\a-\pr_2^*\a$,
  $\pr_i:M\t M\t\R\setminus 0\r M$, $i=1,2$,
  is the projection onto the $i$-th factor, and $t$ is the coordinate in   $\R\setminus 0$.

\begin{thm} \label{thm:LL} \cite{bib:Lyc77},\cite{bib:KM}
If $L$ is a Legendrian submanifold of a contact manifold $(M,\a)$
then there exist an open neighborhood $U$ of $L$ in $M$, an open
neighborhood $V$ of the zero section $0_L$ in $(T^*L\t\R, \a_0)$,
where $\a_0=\theta_L-\dd t$ and $\theta_L$ is the canonical 1-form
on $T^*L$, and a diffeomorphism $\phi:U\r V$ such that
$\phi|_L=\id_L$ and $\phi^*\a_0=\a$.
\end{thm}
Consequently, there is a smooth contactomorphism from a
neighborhood of the
  graph of $(\id_M,1_M)$ onto a neighborhood of zero in the space $J^1(M,\R)$ of 1-jets of
  elements of $\cc(M)$. A Legendrian  submanifold $C^1$-close to
  the graph of $(\id_M,1_M)$ corresponds to the 1-jet of a smooth
  function on $M$ $C^2$-close to zero.

Let $k=0,1\ld n+1$.
 We denote the coordinates in
$\C^m_k\t\rz^{m+1}=(\mathbb S^1)^k\t\rz^{2n-k+1}\t\rz^{2n+1}\t\rz$
by $(\xi_0,\xi,y,\bx_0,\bx,\by,t)$, where we write $\bx=(\bx_1\ld
\bx_n), \by=(\by_1\ld \by_n)$. We identify $\C^m_k\t\rz^{m+1}$
with $T^*L\t\R$, where $L=\C^m_k\t 0\t 0\s\C^m_k\t\R^{m+1}$. Then
for the canonical 1-form $\a_0$ on $T^*L\t\R$ we have
 \begin{equation*}
\a_0=\bx_0\dd \xi_0+\sum_{i=1}^n(\bx_i\dd \xi_i+\by_i\dd y_i)-\dd
t.\end{equation*}

Next,  let $U$ be a small open neighborhood of $L=\C^m_k\t 0\t
0\s\C_k^m\t\rz^m\t\rz$. We have an embedding
$\delta_{\C^m_k}:U\r\widetilde\C^m_k=\C^m_k\t\C^m_k\t\R\setminus
0$ given by
 \begin{equation*}
\delta_{\C^m_k}(\xi_0,\xi,y,\bar x_0,\bar x,\by,t)=
(\xi_0,\xi,y,\xi_0+\bar x_0,\xi+\bar x,y+\by, t+1).
\end{equation*}
Since
\begin{equation*} \widetilde{\as}=t(\dd
\xi_0-\sum_{i=1}^ny_{i}\dd \xi_i)-\dd\bar
x_0+\sum_{i=1}^n\by_{i}\dd\bx_i,
\end{equation*}
we obtain on $U$
\begin{equation*}
\widehat{\ac}:=\delta^*_{\C^m_k}\widetilde{\as}=(t+1)\left(\dd\xi_0
-\sum_{i=1}^ny_{i}\dd \xi_i\right)-\dd(\xi_0+\bar x_0)
+\sum_{i=1}^n(y_i+\by_{i})\dd(\xi_i+\bar x_i).
\end{equation*}
 Then $L$ is a Legendrian submanifold w.r.t. both $\a_0$ and
 $\widehat{\ac}$.

Observe that a diffeomorphism $f$ of $\C_k^m$, $C^1$ and $C^0$
close to the identity, is a contactomorphism iff its graph
\[\Gamma_f(\C_k^m)=\{(p,f(p)-p,\lambda_f(p)-1):p\in\C_k^m\}\]
is a Legendrian submanifold of $(U,\widehat{\as})$.

From now on we will write for $A>0$
\begin{equation}
\tilde E^{(0)}_A=[-A^2,A^2]\t[-A,A]^{2n}\t\R^{m+1},\quad \tilde
E^{(k)}_A=(\mathbb S^1)^k\t [-A,A]^{m-k}\t\R^{m+1},
\end{equation}
where $k=1\ld n+1$. First we consider the case $k=0$. Let
$\phi:U\r V$ be as in Theorem 4.1, where $U\s\R^m\t\R^{m+1}$ is an
open neighborhood of $L$ as above and $\phi^*\a_0=\widehat{\as}$.
 Throughout we set
\begin{equation}
K_{\phi,r}=\sup_{s=0\ld r+1}\max\{\|D^s\phi|_{U\cap \tilde
E^{(0)}_1 }\|,\|D^s(\phi|_{U\cap \tilde E^{(0)}_1})^{-1} \|\}.
\end{equation}
  We have $\forall r\geq 1$,
$K_{\phi,r}<\infty$, as we may assume that $U\cap \tilde
E_1^{(0)}$ is relatively compact.

\begin{prop}  Under the above notation we
have: \begin{enumerate} \item $\phi=\phi_0$ may be chosen so that
it is independent of the variable $x_i$, $i=0,1\ld n$, that is
${\p(\phi-\id)\over \p x_i}=0$.
 \item For any $A>1$ there is a contactomorphism
$\phi_A:U'\r V'$, where $U'$, $V'$ are open neighborhoods  of $L$,
such that $\phi_A|_L=\id_L$, and $\phi_A^*\a_0=\widehat{\as}$.
Moreover, for $r=0,1,\ldots$, one has $K_{\phi,r,A}\leq
A^2K_{\phi,r}$, where \[ K_{\phi,r,A}=\sup_{s=0\ld
r+1}\max\{\|D^s\phi_A|_{U'\cap \tilde
E_A^{(0)}}\|,\|D^s(\phi_A|_{U'\cap \tilde E_A^{(0)}})^{-1}\|\}.\]
\item $\phi_A$ is independent of $x_i$, $i=0,1\ld n$.
\end{enumerate}
\end{prop}
\begin{proof}
(1) We appeal to section 43.18 in [9]. Observe that the contact
forms  $\a_0$ and $\widehat\as$ are independent of the variables
$x_i$, $i=0\ld n$, and $L$ is a Legendrian submanifold w.r.t. both
of them. By an algebraic argument there is a vector bundle
isomorphism $\gamma:T\R^{2m+1}|_L\r T\R^{2m+1}|_L$ such that
$\gamma^*\a_0=\widehat\as$ and $\gamma$ is independent of $x_i$.
Therefore there exists a diffeomorphism $\psi:U\r V$, where $U$,
$V$ are open neighborhoods of $L$ in $\R^{2m+1}$ such that
$\dd\psi|_L=\gamma$, $\psi|_L=\id_L$, and $\psi$ is independent of
$x_i$. Denote $\a_1=\psi^*\widehat\as$. Then $\a_1$ and the
contact form $\a_t=(1-t)\a_0+t\a_1$ existing on a possibly smaller
$V$ are still independent of $x_i$.

Let $f_0=\id$ and let $f_t$, $t\in\R$, be a smooth curve of
diffeomorphisms in an open neighborhood of $L$ such that $\dd
f_t|_{TL}=\id_{T\R^{2m+1}}|_{TL}$. Let $X_t$ be the corresponding
time dependent vector field, i.e. ${\p f_t\over \p t}=X_t\ci f_t$.
It follows that \begin{align*}{\p\over\p t}f_t^*\a_t&={\p\over\p
t}f_t^*\a_s|_{s=t}+f_s^*{\p\over\p
t}\a_t|_{s=t}=f_t^*L_{X_t}\a_t+f_t^*(\a_1-\a_0)\\
&=f_t^*(i_{X_t}\dd\a_t+\dd i_{X_t}\a_t+\a_1-\a_0).\end{align*}
Therefore the proof consists in a construction of $X_t$ such that
\begin{equation}i_{X_t}\dd\a_t+\dd i_{X_t}\a_t+\a_1-\a_0=0
\end{equation} and such that $X_t$ is independent of $x_i$. Indeed, then
$\phi=f_1^{-1}\ci\psi$ satisfies the claim.

We have $\a_0=\a_1$ along $L$ and $X_{\a_0}={\p\over\p t}$ is not
tangent to $L$.   Therefore $X_{\a_t}=X_{\a_0}$ along $L$ and
$X_{\a_t}$ is not tangent to $L$. Consequently, there exists a
submanifold $N$ of codim 1 in $\R^{2m+1}$ containing $L=\R^m\t 0$
such that $N$ is transversal to the flow $\fl^{X_{\a_t}}$ for all
$t\in[0,1]$. Define a time dependent $\R$-valued function $u_t$ by
\[
u_t(\fl^{X_{\a_t}}_s(p))=\int_0^s(\a_1-\a_0)(X_{\a_t})(\fl^{X_{\a_t}}_{\tau}(p))d\tau\]
for $p\in N$. Hence $u_t$ does not depend on $x_i$ and it
satisfies
\begin{equation*}
\dd u_t(X_{\a_t})=i_{X_{\a_t}}(\a_1-\a_0).
\end{equation*}
 Now for the time dependent 1-form $\beta_t=\a_0-\a_1+\dd
u_t-u_t\a_t$, due to the existence of the isomorphism  $I_{\a_t}$,
see (2.2),  there is a unique time dependent vector field $X_t$,
independent of $x_i$, such that $i_{X_{t}}\dd
\a_t+\a_t(X_t)\a_t=\beta_t$. Since $u_t=0$ on $L$ and $\dd
u_t|_{TL}=0$, $f_t$ is defined in a neighborhood of $L$ in
$\R^{2m+1}$ for all $t\in[0,1]$. It follows that $X_t$ satisfies
(4.4).

 (2) Let $\mu_A$, $\nu_A$ be the diffeomorphisms of $\R^{2m+1}$
 given by  \[
\mu_A(x_0,x,y,\bx_0,\bx,\by,t)=(A^2x_0,Ax,Ay,A^2\bx_0,A\bx,A\by,t),\]
and
\[\nu_A(x_0,x,y,\bx_0,\bx,\by,t)=(A^2x_0,Ax,Ay,\bar x_0,A\bx,A\by,A^2t).\]
We have $\mu_A^*\widehat\as=A^2\widehat\as$ and
$\nu_A^*\a_0=A^2\a_0$. That is, $\mu_A$ and $\nu_A$ are
contactomorphisms w. r. t. $\widehat\as$ and $\a_0$, resp., with
$\lambda_{\mu_A}=\lambda_{\nu_A}=A^2$.
 We have $\mu_A(\tilde E^{(0)}_1)=\tilde E^{(0)}_A$ and $\nu_A(\tilde E^{(0)}_1)=\tilde
 E^{(0)}_A$.
 Put $U'=\mu_A(U)$, $V'=\nu_A(V)$, and
 $\phi_A=\nu_A\ci\phi\ci\mu_A^{-1}$. It follows that
$\phi_A^*\a_0=\widehat\as$.
 Since $\|D^s\mu_A\|=\|D^s\nu_A\|=0$ for $s>1$, it is
 apparent from (3.3) that the inequality $K_{\phi,r,A}\leq
A^2K_{\phi,r}$ holds.

 (3)  It is clear by definition that $\phi_A$ is independent of
 $x_i$ if $\phi$ is so.
\end{proof}

\begin{prop} Consider the contact form $\widehat{\ac}$ in a neighborhood of
$L\s\C^m_k\t\R^{m+1}$, $k=1\ld n+1$. Then we
 have:
 \begin{enumerate}
\item For any $A>1$ there is a contactomorphism
$\phi_{A}=\phi_{A,k}:U'\r V'$, where $U'$, $V'$ are open
neighborhoods of $L$,  such that $\phi_{A}|_L=\id_L$,
$\phi_{A}^*\a_0=\widehat{\ac}$, and for $r=0,1,\ldots$ one has
$K_{\phi,r,A}\leq A^2K_{\phi,r}$, where
\[ K_{\phi,r,A}=\sup_{k=1\ld n+1}\sup_{s=0\ld r+1}\max\{\|D^s\phi_A|_{U'\cap
\tilde E^{(k)}_A}\|,\|D^s(\phi_A|_{U'\cap \tilde
E^{(k)}_A})^{-1}\|\},\] c.f. (4.2), (4.3).
 \item $\phi_{A}$ can be chosen so that it
is independent of the variables
 $\xi_i$, $i=0,1\ld n$.
\end{enumerate}
\end{prop}

\begin{proof}
We will apply $\phi_A$ defined in Proposition 4.2. Since it is
independent of $x_i$, it determines uniquely the map $\phi_{A,k}$
which verifies all the requirements.
 \end{proof}

  Following the proof of Theorem 43.19 in
[9] we  can construct a
 chart  at the identity in $\cont_{E}(\C_k^{m},\as)$, where $E$ is a subinterval
of $E_A^{(k)}$, by means of  $\phi_A$ from Propositions 4.2 and
4.3,
 \begin{equation*}
 \Phi_{A}:\cont_{E}(\C_k^{m},\as)\supset\U_1\ni f\mapsto u_f\in\V_2\s
 \CC^{\infty}_{E}(\C_k^m),
 \end{equation*}
 where $\CC^{\infty}_{E}(\C_k^m)$ is the totality of $\R$-valued functions on $\C_k^m$
 compactly supported in $E$.
 Here  $\U_1$ is a $C^1$-neighborhood of the identity in $\cont_{E}(\C_k^{m},\as)$,  $\V_2$
 is a $C^2$-neighborhood of zero in $\CC^{\infty}_{E}(\C_k^m)$, and   $\Phi_{A}(\id)=0_{\C^m_k}$.

\begin{conv}
In the subsequent steps of the proof of Theorem 1.1 the $C^1$
neighborhood $\U_1$ and the $C^2$ neighborhood $\V_2$ will be
possibly shrunk several times and the resulting neighborhoods will
depend on  $r$, $A$, $k$, $\phi$ as above, and a smooth function
$\psi$.

The chart $\Phi_A$ in the proof of Theorem 1.1 will be actually
$\Phi_{A^5}$, so in the sequel we will use in inequalities the
coefficient $\ab$, $\beta$ being a constant, rather than $A^2$,
$A^4$, and so on.
\end{conv}

The construction of $\Phi_A$  is the following. Let $\U_1$ be a
small $C^1$ neighborhood of id in $\cork$. In particular, if
$f\in\U_1$ then $\mu_0^*(f)<\frac{1}{4}$.  For any
$f\in\cont_{E}(\C_k^{m},\as)\cap\U_1$  let
$\Gamma_f=(\id,f-\id,\lambda_f-1):\C^m_k\r\C^m_k\t\rz^{m+1}$ be
the corresponding graph map, that is
 $\Gamma_f(p)=(p,f(p)-p,\lambda_f(p)-1)$ for all $p\in\C^m_k$.
 Then we set
\begin{equation}
\Phi_{A}(f)=
u_f=\pr_3\ci\phi_A\ci\Gamma_f\circ(\pr_1\ci\phi_A\ci\Gamma_f)^{-1},\end{equation}
 where  $\pr_i$ is the projection of
$\C_k^m\t\rz^m\t\rz$ onto the $i$-th factor ($i=1,2,3$), and we
have
\begin{equation}\dd u_f=\pr_2\ci\phi_A\ci\Gamma_f\circ(\pr_1\ci\phi_A\ci\Gamma_f)^{-1},\end{equation}
since $\phi_A\ci\Gamma_f\circ(\pr_1\ci\phi_A\ci\Gamma_f)^{-1}$ is
a section of $\pr_1$ and a Legendre map w.r.t. $\alpha_0$.
Conversely, if $u=u_f\in\V_2$ then
\begin{equation}\Phi_A^{-1}(u)-\id=f-\id=\pr_2\ci\phi_A^{-1}\ci\gd_u\circ(\pr_1\ci\phi_A^{-1}\ci\gd_u)^{-1}
\end{equation}
and
\begin{equation}\lambda_f-1=\pr_3\ci\phi_A^{-1}\ci\gd_u\circ(\pr_1\ci\phi_A^{-1}\ci\gd_u)^{-1}.\end{equation}
Here $\gd_u:\C_k^m\r\C^m_k\t\rz^{m+1}$ is given by
$\gd_u(p)=(p,\dd u(p),u(p))$ for $u\in\CC^{\infty}(\C_k^m)$ and
$p\in\C_k^m$. It is easily seen that $\Phi_A^{-1}$ given by (4.7)
is actually the inverse mapping of $\Phi_A$ given by (4.5).

From now on  for a smooth function $h:\R^{2m+1}\r\R^{2m+1}$ and
$r\geq 1$ we denote by $D^r_{(1)}h$ (resp. $D^r_{(2)}h$) the
totality of partial derivatives of order $r$ w.r.t. the first $m$
variables (resp. the totality of partial derivatives of order $r$
which contain at least one derivative w.r.t. the last $m+1$
variables). Consequently, we can write
\begin{equation}
D^rh=(D^r_{(1)}h, D^r_{(2)}h).
\end{equation}

\begin{lem} Suppose $r\geq 2$ and $k=0,1\ld n+1$ . Under the  notation of Propositions 4.2 and 4.3,
 there are   constants $\beta$ and $C_{\phi,r}$,  and $U'=U_{\phi,r,A}$, an open
neighborhood of $L$, such that for $i=1,2,3$
\[|D^r_{(1)}(\pr_{i}\ci\phi_A)(p)|\leq  \ab\cfr |p_2|\]
 for any $p\in U'\cap \tilde E_A^{(k)}$. Here  we denote $p=(p_1,p_2)\in\C^m_k\t\rz^{m+1}$.
  The same is true if $\phi_A$ is replaced by $\phi_A^{-1}$.
\end{lem}
\begin{proof}
Observe that $D^r_{(1)}(\pr_i\ci\phi_A)$ is a locally Lipschitz
map and that in view of Propositions 4.2 and 4.3 the Lipschitz
constant may be written in the form $A^2\cfr $, since
$\phi_A|_L=\id_L$ and $D^r_{(1)}(\pr_i\ci\phi_A)=0$ on $L$ by
definition of $\phi_A$. Consequently,
$D^r_{(1)}(\pr_i\ci\phi_A)|_{U'\cap \tilde E_A^{(k)}}$ is
Lipschitz , where $U'$ is an open neighborhood  of $L$ . The same
is true for $\phi_A^{-1}$. This implies the lemma.
\end{proof}

\begin{prop} Let $E$ be a subinterval of $E_A^{(k)}$.
Under the above notation,  for any $r\geq 2$ there is  a $C^1$
neighborhood $\U_{1}$ of the identity in $\cont_{E}(\C^m_k,\as)$
such that for any $f\in\U_{1}$  one has\begin{enumerate} \item
$\|D^{r+1}u_f\|\leq C_{\phi} \mu^{\l}_r(f)+\ab
P_{\phi,r}(M^{\l}_{r-1}(f))$, \item $\mu^{\l}_r(f)\leq C_{\phi}
\|D^{r+1}u_f\|+\ab P_{\phi,r}(\sup_{i=0\ld
r}\|D^iu_f)\|)$,\end{enumerate} where $P_{\phi,r}$ has no constant
term and $\beta$, $C_{\phi}$ are constants.
\end{prop}
\begin{proof}
Set $\phi_1=\pr_1\ci\phi_A$, $\phi_2=\pr_2\ci\phi_A$.

 (1) By (4.5), (4.6) , (3.3) and
Propositions 4.2 and 4.3 we have for $2\leq s\leq r$
\begin{equation}
\|D^s(\phi_2\g_f)\| \leq C_{\phi} \mu^*_{s}(f)+\ab
P_{\phi,s}(M^*_{s-1}(f)).\end{equation} In fact, the only
nontrivial thing is to estimate $\|(D^s\phi_2
\ci\g_f)\cdot(D\g_f\t\cdots\t D\g_f)\|$ but, due to decomposition
(4.9) and Lemma 4.5, we have
\begin{align*}
\|(D^s\phi_2 \ci\g_f)\cdot(D\g_f\t\cdots\t
D\g_f)\|&\leq\|D^s_{(1)}\phi_2\ci\g_f\|+\|D^s_{(2)}\phi_2\ci\g_f\|\mu_1^*(f)\\
&\leq \ab C_{\phi,s}'\mu^*_{0}(f)+C''_{\phi,s}\mu^*_1(f)\\
&\leq\ab C_{\phi,s}M^*_{s-1}(f).
\end{align*}
We have
\begin{align*}
\|D^{r+1}u_f\|&=\|D^r(Du_f)\|=\|D^r(\phi_2\ci\Gamma_f\circ(\phi_1\ci\Gamma_f)^{-1})\|\\
&\leq C_{\phi} \mu^*_{r}(f)+\ab
P_{\phi,r}(M^*_{r-1}(f)).\end{align*} Indeed, in view of (3.3),
(3.5), (4.10) and Lemma 3.6(4), denoting
$\phi_{1f}=\phi_1\ci\Gamma_f$, the only nontrivial term to
estimate is
\begin{equation*}
\|D\phi_{1f}^{-1}\cdot(((D^r\phi_1\circ\Gamma_f)\cdot(D\Gamma_f\t\cdots\t
D\Gamma_f))\circ\phi_{1f}^{-1}) \cdot(D\phi_{1f}^{-1}\t\cdots\t
D\phi_{1f}^{-1})\|,
\end{equation*}
and this can be obtained as above.

 (2) We proceed analogously as
in (1) and, in addition, we have to show that
\[\|D^r(\lambda_f-1)\|\leq C_{\phi}
\|D^{r+1}u_f\|+\ab P_{\phi,r}(\sup_{i=0\ld r}\|D^iu_f)\|).\] This
can be done as above in view of (4.8), (4.9) and Lemma 4.5.
\end{proof}

\section{Two kinds of  fragmentations }

In most papers on the simplicity and perfectness of diffeomorphism
groups a clue role is played by fragmentation properties. These
properties  enable usually to reduce the proof to the case
$M=\rz^m$. Contrary to the volume element case  and the symplectic
case (c.f. [2]), in the contact case the fragmentation property
takes its general form.

 The following fragmentation
property for infinitesimal contact automorphisms is a consequence
of Proposition 2.1.

\begin{lem} Let $X\in \X_{c }(M,\a)$ with $\supp(X)\s\bigcup
^k_{i=1}U_i$, where $U_i$ are  open. Then there is a decomposition
$X=X_{1}+\cdots +X_{k}$ such that $X_{i}\in\X_{c } (M,\a) $ and
$\supp(X_{i})\subset U_i$. The same is true for smooth curves in
$\X_{c } (M,\a) $ instead of elements of $\X_{c } (M,\a)$.
\end{lem}

It follows the fragmentation property for $\con$.

\begin{lem} Let $f\in \con$ and let  $\{U_i\}_{i=1}^k$ be an open
cover of $M$.
 Then there exist   $f_{j}\in
\con$, $j=1\ld l$, with $f=f_{1}\ci\ldots\ci f_{l}$  such that
$\supp(f_{j})\subset U_{i(j)}$ for all $j$. The same is true for
isotopies of contactomorphisms instead of contactomorphisms.
\end{lem}

The proof exploits the correspondence between isotopies in $\con$
and smooth curves in  $\X_{c } (M,\a)$ given by (2.1) combined
with Lemma 5.1.

The fragmentation in Lemma 5.2 is said to be of the {\it first
kind}. This lemma enables to replace $\con$ by $\corr$ in the
proof of Theorem 1.1. However, we need in this proof also the {\it
second kind } of fragmentations. Such  fragmentations exist in a
$C^1$ neighborhood of the identity in the groups $\cork$,
$k=0,1\ld n+1$. Moreover, we  claim that the norms of the factors
of a given fragmentation are estimated by the norm of the initial
contactomorphism in a convenient way and that the fragmentation
itself is uniquely determined.

\begin{dff} Suppose $E$ is a subinterval of $E_A^{(k)}$.
Let $\psi:\C_k^m\r [0,1]$ be a smooth function. It follows from
Proposition 4.6 that there exists a $C^1$-neighborhood of the
identity $\U_{\phi,\psi,A}\s\U_1$ such that for any
$f\in\U_{\phi,\psi,A}$ with $\supp(f)\s E$ the contactomorphism
\[f^{\psi}:=\Phi_A^{-1}(\psi\Phi_A(f))=\Phi_A^{-1}(\psi
u_f)\]
 is well-defined and $\supp(f^{\psi})\s E$.
\end{dff}
In fact, for any $r\geq 1$ there is a polynomial  without constant
term $P_{\psi,r}$ such that for all $u\in\cc_c(\C_k^m)$

\begin{align}
\begin{split}
\|D^{r+1}(\psi u)\|
&\leq\|D^{r+1}u\|+\sum_{j=1}^{r+1}C_{r,j}\|D^j\psi\|\|D^{r+1-j}u\|\\
&\leq\|D^{r+1}u\|+P_{\psi,r}\big(\sup_{s=0\ld r}\|D^su\|\big).
\end{split}
\end{align}
In particular we may ensure that $\psi u_f\in\V_2$. The following
is obvious.

\begin{prop}
One has $\supp(\fp)\s\supp(\psi)$ and $\fp=f$ on any open
$U\s\C_k^m$ such that $\psi=1$ on $U$.
\end{prop}

\begin{lem} Under the above notation,
  for any $r\geq 2$ there are
 polynomials $P_{\phi,\psi,r}$  without constant term
and  constants $\beta$, $C_{\phi,\psi}$ such that
\begin{equation*}
\mu^*_r(\fp)\leq C_{\phi,\psi} \mu^*_r(f)+\ab
P_{\phi,\psi,r}(M^{\l}_{r-1}(f)),
\end{equation*} whenever $f\in\U_{\phi,\psi,A}$ and
$\supp(f)\s E$. In particular, if $R_{E}\leq 2$ (c.f. (3.1) and
(3.8)) there exists a constant $ C_{\phi,\psi,r}$ such that
$\mu^*_r(\fp)\leq \ab C_{\phi,\psi,r} \mu^*_r(f)$ for all
$f\in\U_{\phi,\psi,A}$ with $\supp(f)\s E$.
 \end{lem}
\begin{proof}
The first assertion follows from Proposition 4.6 and (5.1). The
second is a consequence of Proposition 3.5.
\end{proof}

In particular,   we obtain fragmentations of the second kind on
large intervals in $\R^m$.
\begin{prop}

 Let $2A>1$ be an even integer, let  and
let $\psi:[0,1]\r[0,1]$ be a smooth function such that $\psi=1$ in
a neighborhood of $[0,{1\over 4}]$ and $\psi=0$ on $[{3\over
4},1]$. Then there exists a $C^1$-neighborhood $\U_{\phi,\psi,A}$
of the identity in $\cont_{E_{2A}}(\rr,\as)_0$, c.f. (2.6), such
that for any $f\in\U_{\phi,\psi,A}$ there exists a decomposition
$f=f_1\ldots f_{4A+1}$, uniquely determined by $\phi$, $\psi$ and
$A$,  where each $\supp(f_{\kappa})$ is contained in an interval
of the form $([k-{3\over 4},k+{3\over 4}]\t\rz^{2n})\cap E_{2A}$,
 with $k\in\mathbb Z$ , $|k|\leq 2A$, and
  the inequalities
  \begin{enumerate}
  \item
$\mu_r^*(f_{\kappa})\leq C_{\phi,\psi}\mu^{\l}_r(f)+\ab
P_{\phi,\psi,r}(M^{\l}_{r-1}(f))$, \item $\mu^*_r(f_{\kappa})\leq
\ab C_{\phi,\psi,r} \mu^*_r(f)$, whenever $\supp(f)\s E\s E_{2A}$
with $R_{E}\leq 2$,
\end{enumerate}
hold for all $\kappa=1\ld 4A+1$ and $r\geq 2$.
   Analogous decompositions can be obtained w.r.t. the variables
   $x_i$ and $y_i$, $i=1\ld n$.
\end{prop}

\begin{proof}  By abusing the notation we extend $\psi$ to the
function $\psi:[-1,1]\r[0,1]$ given by $\psi(x)=\psi(-x)$ on
$[-1,0]$, and, finally, to the periodic function
 $\psi:\R\r [0,1]$ of period $2$.

Let $\psi_1=\psi\ci\pr_{1}:\rz^m\r[0,1]$, where
$\pr_1(x_0,x,y)=x_0$. Let $f\in\cont_{E_{2A}}(\rr,\as)_0$ be
sufficiently $C^1$-close to the identity and let $f^{\psi_1}$ be
defined as in Definition 5.3. Then we have \begin{enumerate}\item
$ f^{\psi_1}=\prod_{k=-{A}}^{A}f_{2k}$, with
$\supp(f_{2k})\s[2k-\frac{3}{4},2k+\frac{3}{4}]\t\R^{2n}$, and
\item $f(f^{\psi_1})^{-1}=\prod_{k=-{A}}^{{A}-1}f_{2k+1}$ with
$\supp(f_{2k+1})\s[2k+\frac{1}{4} ,2k+\frac{7}{4})]\t\R^{2n}$.
\end{enumerate} The inequalities follow from Lemmas 3.6 and 5.5.
For convenience we renumerate $f_{\kappa}$.

\end{proof}
By applying Proposition 5.6 consecutively to all variables we get
\begin{prop}
Under the above assumptions , there exists a $C^1$-neighborhood of
the identity $\U_{\phi,\psi,A}\s\cont_{E_{2A}}(\rr,\as)_0$ such
that for any $f\in\U_{\phi,\psi,A}$ there exists a decomposition
$f=f_1\ldots f_{a_m}$, uniquely determined by $\phi$, $\psi$ and
$A$, where $a_m=(4A+1)^{m}$ and where each $\supp(f_{\kappa})$ is
contained in an interval of the form $([k_1-{3\over 4},k_1+{3\over
4}]\t\cdots\t[k_{m}-{3\over 4},k_{m}+{3\over 4}])\cap E_{2A}$,
with $k_i\in\mathbb Z$ , $| k_i|\leq 2A$, for $i=1\ld m$.
Moreover, for all $\kappa=1\ld a_m$ and $r\geq 2$
  \begin{enumerate}
  \item
$\mu_r^{\l}(f_{\kappa})\leq
C_{\phi,\psi}\mu^{\l}_r(f)+A^{\beta(m)}P_{\phi,\psi,r}(M^{\l}_{r-1}(f))$,
\item $\mu^*_r(f_{\kappa})\leq A^{\beta(m)}C_{\phi,\psi,r}
\mu^*_r(f)$, whenever $\supp(f)\s E\s E_{2A}$ with $R_{E}\leq 2$.
\end{enumerate}
\end{prop}

\section{Shifting supports of contactomorphisms}

From now on we set for $A>1$
\begin{equation}
I_A=[-2,2]^{n+1}\t[-2A,2A]^n.
\end{equation}

 In this section we will describe the procedure of
shifting supports of contactomorphisms on $\R^m$ in the $y_i$
directions. Fortunately, this can be done by using the
contactomorphisms $\sigma_{i,t}$, $i=1\ld n$, introduced in
section 2. Fix  $1\leq i\leq n$ and put $\sigma_t=\sigma_{i,t}$.
Recall that $\sigma_t(x_0,x,y)=(x_0+tx_i,x,y+t{\bf 1}_i)$. Notice
that for any $t\in\R$ we have
 $\|D\sigma_t\|=1+|t|$,   and
$\|D^r\sigma_t\|=0$ for all $r>1$. Next we define
$\rho_{A,t}=\eta_A\ci\chi_A\ci\sigma_t$, see section 2.

Under  the assumption $A>5n$, observe that
\begin{equation}
\supp(\rho_{A,t}\ci f\ci\rho_{A,t}^{-1})\s J_A,
\end{equation}
where \begin{equation} J_A=[-A^5,A^5]^{n+1}\t[-2A,2A]^n,
\end{equation}
for all $f\in\corr$ with support in
$[-2,2]^{n+i}\t[k-1,k+1]\t[-2,2]^{n-i}$ with $|k|\leq 2A$ and
suitable $t$. Likewise, the inclusion (6.2) holds for any
$f\in\cont_{I_A}(\rz^m,\as)_0$ with $\supp(f)\s
\rz^{n+1}\t[k_1-1,k_1+1]\t\ldots\t[k_n-1,k_n+1]$ and $|k_i|\leq
2A$, where $i=1\ld n$, with $\rho_{A,t}$ replaced by
$\tilde\rho_{A,\bf t}$ given by
\begin{equation*}
\tilde\rho_{A,\bf t}=\eta_A\ci\chi_A\ci\tilde\sigma_{\bf t},
\end{equation*}
where ${\bf t}=(t_1\ld t_n)$, and $\tilde\sigma_{\bf
t}=\sigma_{1,t_1}\ci\cdots\ci\sigma_{n,t_n}$,
 with suitably chosen $t_i$ so that $|t_i|\leq
2A$ for $i=1\ld n$.

\begin{prop} If $|t_i|\leq 2A$ for $i=1\ld n$ and  $f\in\cont_{I_A}(\rr, \as)_0$
then, for any $r\geq 2$
\begin{equation*}
\mu_r^{\l}(\tilde\rho_{A,\bf t}\ci f\ci\tilde \rho_{A,\bf
t}^{-1})\leq A^{4-r}(3n)^{r+1} \mu_r^*(f).
\end{equation*}
\end{prop}

\begin{proof}
We have
\begin{equation*}
(\rho_{A,t}^{-1})(x_0,x,y)=\sigma_{-t}(A^{-3}x_0,A^{-2}x,
A^{-1}y)=(A^{-3}x_0-tA^{-2}x_i,A^{-2}x, A^{-1}y-t{\bf 1}_i).
\end{equation*}
It follows that $\|D\tilde\rho_{A,\bf t}^{-1}\|\leq 3nA^{-1}$, as
$|t|\leq 2A$. Likewise $\|D\tilde\rho_{A,\bf t}\|\leq 3nA^4$.
Therefore for $r\geq 2$
\begin{align*} \mu_r^{\l}(\tilde\rho_{A,\bf t}\ci f\ci\tilde\rho_{A,\bf t}^{-1})&
\leq\|D\tilde\rho_{A,\bf t}\|\|D^rf\|\|D\tilde\rho_{A,\bf t}^{-1}\|^r\\
 \quad &\leq 3nA^4\|D^rf\|(3nA^{-1})^r \leq
A^{4-r}(3n)^{r+1}\mu^{\l}_r(f),
\end{align*}
in view of (3.3) and the fact that $D^s\tilde\rho_{A,\bf t}=0$
whenever $s>1$.

 Next, notice that $\lambda_{\chi_A}=A^2$, $\lambda_{\eta_A}=A$ and
 $\lambda_{\sigma_t}=1$. Consequently, $\lambda_{\tilde\rho_{A,\bf
 t}}=A^3$ and by (3.6) $\lambda_{\tilde\rho_{A,\bf t}\ci f\ci\tilde\rho_{A,\bf
 t}}=\lambda_f\ci\tilde\rho_{A,\bf t}$
It follows from (3.3) that
 \begin{align*}
\|D^r\lambda_{\tilde\rho_{A,\bf t}\ci f\ci\tilde\rho_{A,\bf
t}^{-1}}\| &=\|D^r(\lambda_f\ci\tilde\rho_{A,\bf t}^{-1})\|\leq
\|D^r\lambda_f\|\|D\tilde\rho_{A,\bf t}^{-1}\|^r\\
& \leq \|D^r\lambda_f\|(3nA^{-1})^r\leq A^{-r}(3n)^r\mu_r^{\l}(f).
\end{align*}
Combining the above inequalities we obtain the claim.
\end{proof}

\section{Construction of a correcting contactomorphism on $\C_k^m$}

In this section for any sufficiently  $C^1$ small contactomorphism
on $\C^m_{k+1}$, $k=0\ld n$,  we construct a correcting
contactomorphism which is indispensable  in the construction of
auxiliary rolling-up operators $\Psi^{(k)}_A$ (Proposition 8.5).
The reason is that, given $f\in\cont_{E^{(k)}_A}(\C^m_k,\as)_0$,
we wish to ensure that the norm $\mu^*_r(\Psi_A^{(k)}(f))$ of the
rolled-up contactomorphism $\Psi_A^{(k)}(f)$ would be controlled
by $\mu^*_r(f)$. The procedure of rolling-up contactomorphisms
will be described in the next section.

 For a smooth function $h:\C_{k+1}^m\r\R^l$ by $D_{[k]}^rh$ we
denote the system of all partial derivatives of order $r$ of $h$
with at least one derivative w.r.t. $\xi_k$.

\begin{lem} Let $h\in\cc(\C^m_{k+1},\rz^l)$.  Then we have
$\|D_{[k]}^sh\|\leq \|D_{[k]}^rh\|$ for all $1\leq s\leq r$.
\end{lem}
\begin{proof}  Let $h=(h_1\ld h_l)$ and $\gamma=(\gamma_0\ld\gamma_{2n})\in\mathbb N_0^m$
with $|\gamma|=s$ and $\gamma_k>0$. Set $\bar\gamma=\gamma +{\bf
1}_k$. Then $|\bar\gamma|=s+1$ and we may integrate
$D^{\bar\gamma}h_i$, $i=1\ld l$, w.r.t. $\xi_k$ and use the fact
that $D^{\gamma}h_i$ vanishes at a point $(\xi_0,\xi, y)$ for any
fixed $(\xi_0\ld\xi_{k-1},\xi_{k+1}\ld y_n)$ to obtain
$\|D^{\gamma}h_i\|\leq\|D^{\bar\gamma}h_i\|$. It follows
$\|D_{[k]}^sh\|\leq \|D_{[k]}^{s+1}h\|$. The claim follows by
induction.
\end{proof}
However, Lemma 7.1 does not hold for $s=0$.

 Observe that we may lift uniquely any $g\in\cont_{
E_A^{(k+1)}}(\C^m_{k+1},\ac)_0$ sufficiently $C^1$ close to the
identity to a contactomorphism $\tilde
g\in\cont_{E_A^{(k)}}(\C^m_k,\as)_0 $ which is periodic with
period 1 ( that is, $\tilde g-\id$ is periodic as a function with
period 1) w.r.t. the variable $\xi_k$. Notice that $\tilde g$
depends continuously on $g$ and $\mu_{r}^{\l}(\tilde
g)=\mu^{\l}_{r}(g)$.

Let us denote for $l=1\ld n+1$
\begin{equation}
\cork^{(l)}=\{f\in\cork: D_{\xi_i}(f-\id)=0,\,i=0,1\ld l-1\},
\end{equation}
where $D_{\xi_i}={\p\over\p \xi_i}$.
 That is, $\cork^{(l)}$ is
the subgroup of $\cork$ consisting of all its elements which are
independent of $\xi_i$, $i=0,1\ld l-1$. Further, denote
$\cork^{(0)}=\cork$.

 Let $f\in\cont_{ E^{(k+1)}_A}(\C^m_{k+1},\ac)_0^{(k)}$ will be sufficiently $C^1$ close to
the identity. By using the chart $\Phi_A$, we put $u_f=\Phi_A(f)$.
Then $u_f\in\cc_c(\C^m_{k+1})$ is independent of $\xi_0\ld
\xi_{k-1}$ in view of Proposition 4.3. Define
$v_f\in\cc_c(\C^m_{k+1})$, independent of $\xi_0\ld \xi_{k}$, by
fixing $\xi_k$ to be equal to 0, that is
\[v_{f}(\xi_{k+1}\ld \xi_n,y)=u_f(0,\xi_{k+1}\ld \xi_n,y).\]  Then $u_f,
v_f\in\V_2$ and we define
\begin{equation}\hat f=\Phi_A^{-1}(v_f).\end{equation}
Notice that $\hat f$ is independent of $\xi_0\ld \xi_{k}$ as $v_f$
is so. Next we put \begin{equation}
 w_f=\Phi_A(f\hat f^{-1}).
\end{equation}

 Observe that the equality
\[\hat f(\xi_{k+1}\ld \xi_n,y)=f(0,\xi_{k+1}\ld \xi_n,y),\] is not
true, since $\hat f$ defined by it does not fulfil the equalities
in Proposition 2.2, provided $f$ does.

Finally, denote for $r\geq 1$
\begin{equation*}
\nu^*_r(f)= C_{\phi}K^r\mu_{r}^*(f)+F_{\phi,r}(M^*_{r-1}(f)),
\end{equation*}
where $K$, $C_{\phi}$  are constants,
  $F_{\phi,r}$ is an admissible polynomial and $F_{\phi,1}=0$.

\begin{prop} Let  $ E$ be a subinterval of $ E^{(k+1)}_A\s\C^m_{k+1}$, $k=0,1\ld
n$. There exist constants and polynomials as above and a constant
$\beta$ such that if $f$ belongs to a sufficiently small $C^1$
neighborhood $\U_1=\U_{\phi,A}$ of the identity in $\cont_{
E}(\C^m_{k+1},\ac)_0^{(k)}$ then
 we have for all $r\geq 1$:
\begin{enumerate}
 \item $\hat f\in\cont_{ E}(\C^m_{k+1},\ac)_0^{(k+1)}$ and $\lambda _{\hat f}=1$.

  \item
$\forall 1\leq s\leq r+1$,\quad $\|D_{[k]}^su_f\|\leq
C_{\phi}\mu^*_r(f) $.

 \item
$\forall 0\leq s\leq r$,\quad $\mu^*_s(f\hat f^{-1})\leq
\nu^*_r(f)$.

\item  $\mu^*_r(\hat f)\leq \nu^*_r(f)$.

\item $\forall 0\leq s\leq r+1$,\quad $\|D^sw_f\|\leq
\ab\nu^*_r(f)$.
\end{enumerate}
\end{prop}
\begin{proof}
For short we will write $\phi_{i}=\pr_i\ci\phi_A$ and
$\bar\phi_{i}=\pr_i\ci\phi_A^{-1}$ for $i=1,2,3$, that is
$\phi_A=(\phi_{1},\phi_{2},\phi_{3})$ and
$\phi_A^{-1}=(\bar\phi_{1},\bar\phi_{2},\bar\phi_{3})$, c.f.
(4.5)-(4.8). Further we will denote $\phi_{if}=\phi_i\ci\Gamma_f$,
$\bar\phi_{iu}=\bar\phi_i\ci\gd_u$, $i=1,2,3$.

Let $\mathcal M_m$ be the set of all nonsingular  matrices of deg
$m$. By the Lipschitz property of the inverse mapping in $\mathcal
M_m$ there are a neighborhood $U$ of id in $\mathcal M_m$ and a
constant $L$ such that for all $m_1,m_2\in U$
\begin{equation}
|m_1^{-1}-m_2^{-1}|\leq L|m_1-m_2|.
\end{equation}

(1) As we stated above $\hat f$ is independent of $\xi_0\ld
\xi_k$.
 Since $X_{\ac}={\p\over\p \xi_0}$, we have $\hat f_*X_{\ac}=X_{\ac}$. Consequently, by
(4.1), $\lambda_{\hat f}=\hat f^*(i_{X_{\ac}}\ac)=1$.

(2) First note that $D_{[k]}\phi_{1f}={\bf 1}_k +l_f$, where for
any $p\in\C^m_k $ $l_f(p)\in\R^m$ is such that $\|l_f\|\leq
C_{\phi}\|D_{[k]}(f-\id)\|$. Then by Lemma 7.1 $\|l_f\|\leq
C_{\phi}\mu^*_r(f)\leq\nu^*_r(f)$. Thanks to (7.4) and the formula
for inverse matrix, the same property possesses
$D_{[k]}\phi_{1f}^{-1}$.

Observe that
\begin{equation} (D_{\xi_k}\phi_A)_i=\delta_{ik},\quad i=1\ld 2m+1,
\end{equation}
due to Proposition 4.3.
 It follows that
\begin{align*}\|D_{[k]}\phi_{3f}\|
&=\|(D\phi_3\ci\g_f)\cdot D_{[k]}\g_f\| \leq
\|D_{\xi_k}\phi_3\|+\|D\phi_3\|\|D_{\xi_k}(f-\id)\|\\
&=C_{\phi}\|D_{\xi_k}(f-\id)\| \leq
C_{\phi}\mu^*_r(f)\leq\nu^*_r(f),
\end{align*}
 by using  (7.5) and Lemma 7.1. Now, due to (4.5), (7.4) and the above
 arguments,
  \begin{equation*}\|D_{[k]}u_f\|=\|(D\phi_{3f}\ci\phi_{1f}^{-1})\cdot
  D_{[k]}\phi_{1f}^{-1}\|\leq\|D_{[k]}\phi_{3f}\|+\|l_f\|\leq
C_{\phi}\mu^*_r(f)\leq\nu^*_r(f),
\end{equation*}
where $l_f$ corresponds to $D_{[k]}\phi_{1f}^{-1}$.

 For $s>1$ we have $D^s_{[k]}\phi_A=0$ by (7.5). We use (4.6), (3.3), (3.5), (7.4), and the
 proof is similar.

(3) From the definition of $v_f$ we have \begin{equation}
\|u_f-v_f\|\leq\|D_{\xi_k}u_f\|,\quad
\|Du_f-Dv_f\|=\|D_{\xi_k}u_f\|
\end{equation}
  First we show (3) for $s=0$. Observe that $\|f\hat
 f^{-1}-\id\|= \|f-\hat f\|$. By (4.7) we obtain
 \begin{align*}
\|f-\hat f\|
&=\|\bar\phi_{2u_{f}}\bar\phi_{1u_{f}}^{-1}-\bar\phi_{2v_{f}}\bar\phi_{1v_{f}}^{-1}\|\\
&=\|\bar\phi_{2u_{f}}\bar\phi_{1u_{f}}^{-1}-\bar\phi_{2u_{f}}\bar\phi_{1v_{f}}^{-1}\|+
\|\bar\phi_{2u_{f}}\bar\phi_{1v_{f}}^{-1}-\bar\phi_{2v_{f}}\bar\phi_{1v_{f}}^{-1}\|\\
&\leq L_{\phi}\|\gd_{u_{f}}-\gd_{v_{f}}\| \leq C_{\phi}\mu^*_r(f),
\end{align*}
due to (2), (7.6) and the Lipschitz property.

 Next, in view of (1) and
(3.6) we have $\|\lambda_{f\hat
f^{-1}}-1\|=\|\lambda_f-\lambda_{\hat f}\|$. Hence by (4.8) and a
similar argument, $\|\lambda_{f\hat f^{-1}}-1\|\leq
C_{\phi}\mu^*_r(f)$.

By (4.7) and (3.2) we get
\begin{align}
\begin{split}
D(f-\hat f)=&(D\bar\phi_{2u_f}\ci\bar\phi_{1u_f}^{-1})\cdot
D\bar\phi_{1u_f}^{-1}-(D\bar\phi_{2v_f}\ci\bar\phi_{1v_f}^{-1})\cdot
D\bar\phi_{1v_f}^{-1}\\
 =&\big((D\bar\phi_{2u_f}\ci\bar\phi_{1u_f}^{-1})\cdot
D\bar\phi_{1u_f}^{-1}-(D\bar\phi_{2u_f}\ci\bar\phi_{1v_f}^{-1})\cdot
D\bar\phi_{1v_f}^{-1}\big)\\
&+\big((D\bar\phi_{2u_f}\ci\bar\phi_{1v_f}^{-1})\cdot
D\bar\phi_{1v_f}^{-1}-(D\bar\phi_{2v_f}\ci\bar\phi_{1v_f}^{-1})\cdot
D\bar\phi_{1v_f}^{-1}\big).
\end{split}
\end{align}
It follows that $\|D(f-\hat f)\|\leq C_{\phi}\mu^*_r(f)$, due to
(2), (7.6)  and the Lipschitz property.

Let $1< s\leq r$. In view of (3.3), (7.6), (7.7), the Leibniz
rule, the Lipschitz property and, again, (2) we have
\begin{equation}
\|D^s(f-\hat{f})\|\leq \sup_{|\gamma|=s-1}\|D^{\gamma}(D(f-\hat
f))\|\leq \nu^*_r(f).
\end{equation}
 Likewise
 \begin{equation}
\|D^s(\lambda_f-\lambda_{\hat f})\|\leq \nu^*_r(f).
\end{equation}

Now, since we have \begin{align*} D^s(f\hat
{f}^{-1}-\id)&=D^{s-1}(D(f\hat {f}^{-1}-\id))\\
&=D^{s-1}\big((D f\ci \hat f^{-1})\cdot D\hat f^{-1}- (D \hat
f\ci\hat f^{-1})\cdot D \hat f^{-1}\big),\end{align*} (3) for
$s\geq 1$ follows  from (3.3), the Leibniz rule, (7.8) and (7.9).

 (4) It is an immediate consequence of (7.8) and (7.9).

(5) To simplify notation let $g= f\hat f^{-1}$. In view of (4.5)
and (7.3) we have
$w_f=\phi_3\g_g(\phi_1\g_g)^{-1}=\phi_{3g}\phi_{1g}^{-1}$. For
$1\leq s\leq r$ we have
\begin{align*}\|D^{s+1}\phi_{3g}\|&\leq
C_{\phi}(\mu^*_s(g)+A^2\mu^*_0(g))\\
&\quad+C_{\phi,s}\sup\|(D^i\phi_3\ci
\g_g)\cdot(D^{j_1}\g_g\t\cdots
D^{j_i}\g_g)\|\\
&\leq\ab\nu^*_r(f),
\end{align*}
 where $\sup$ is taken over $i=2\ld s-1$, with
$j_1+\cdots j_i=s$, $j_l\geq 1$ for $l=1\ld i$, and $j_l>1$ for
some $l$. In fact, it follows from (3) and (4) above, (3.3) and
Lemma 4.5. In order to obtain (5) for $s\geq 2$, in view of (3.3)
and (3.5), it suffices to show that
$\|D^s\phi_{1g}\|\leq\ab\nu^*_r(f)$, and this can be done
analogously as above.

Finally, to obtain (5) for $s=0$ and $s=1$ we integrate $D^2w_f$
w.r.t. $y_1$ twice or once, bearing in mind that $\supp(w_f)\s
E^{(k+1)}_A$.
\end{proof}

\begin{cor}
If $f$ belongs to a sufficiently small $C^{r-1}$ neighborhood
$\U_1=\U_{\phi,r,A}$ of the identity in $\cont_{
E}(\C^m_{k+1},\ac)_0^{(k)}$ then
 we have for all $r\geq 1$:
\begin{enumerate}
 \item  $\mu^*_r(\hat f)\leq C_{\phi,r}\mu^*_r(f)$.

\item $\forall 0\leq s\leq r+1$,\quad $\|D^sw_f\|\leq \ab
C_{\phi,r}\mu^*_r(f)$.
\end{enumerate}

\end{cor}
In fact, we can rewrite (4) as
\[\mu^*_r(\hat f)\leq
C_{\phi}K^r\mu_{r}^*(f)+F_{\phi,r,A}(M^*_{r-1}(f)),\] and use
Definition 3.2. Similarly, we can proceed with (5).

\section{Rolling-up contactomorphisms}

A possible application of  Mather's rolling-up operators
$\Psi_{i,A}$ (cf.[12]) to the contact case fails completely in the
$y_{i}$ directions. But even in the ''good'' directions $x_i$,
$i=0\ld n$, the operators $\Psi_{i,A}$ do not apply verbatim. The
next and greater difficulty is that for a contactomorphism $f$ the
class $[\Psi_{i,A}(f)]$ need not be equal to $[f]$ in the
abelianization $H_1(\corr)$. Roughly speaking, the reason is that
given $f\in\cont(\R^m,\as)_0$ with $\supp(f)\s \R\t[-A,A]^{2n}$,
any $g\in\corr$ such that $g=f$ on $[-A,A]^m$ must depend on
$y_i$, c.f. (2.3). This fact seems to spoil any possible proof
that $[\Psi_{0,A}(f)]=[f]$, and the same is for $i=1\ld n$.

In the present section we define a new rolling-up operator which
works in the contact category (Proposition 8.7). To this end we
will use the contact cylinders $(\C^m_k,\as)$, $k=1\ld n+1$. The
correcting contactomorphisms defined in the previous section
enable us to define auxiliary rolling-up operators $\Psi_A^{(k)}$
acting on $\cork$. A clue observation is that a "remainder"
contactomorphism living on $\C^m_{n+1}$ admits a representant in
the commutator subgroup of $\corr$.

 Observe that the application of
the rolling-up operator is indispensable in the proof. In fact, we
cannot apply the procedure described in section 5 (the
fragmentation of the second kind) to the group
$\cont_{J_A}(\rr,\as)_0$,  considered in the proof of Theorem 1.1
(section 9), since in this case a coefficient of the form $A^{Cr}$
would appear in Proposition 5.7(2) and the proof would be no
longer valid.

In this section $A$ is a large positive integer. Throughout we
denote
 \begin{align}
 \begin{split}
 J_A^{(0)}&=J_A=[-A^5,A^5]^{n+1}\t[-2A,2A]^n,\\
 J_A^{(k)}&=(\mathbb S^1)^k\t[-A^5,A^5]^{n-k+1}\t[-2A,2A]^n, \quad k=1\ld n,\\
K_A^{(0)}&=K_A=[-2,2]\t[-A^5,A^5]^{n}\t[-2A,2A]^{n},\\
K_A^{(k)}&=(\mathbb
S^1)^k\t [-2,2]\t[-A^5,A^5]^{n-k}\t[-2A,2A]^{n},\quad k=1\ld n.\\
\end{split}
\end{align}
  Observe that $R_{K^{(k)}_A}=2$ and $R_{J^{(k)}_A}=2A$ (c.f. (3.1) and (3.8)).

Denote by $\pi_k:\C^m_k\r\C^m_{k+1}$, $k=0,1,\ld n$, the canonical
projection. In other words, $\pi_k$ is induced by the canonical
projection $\pi:\R\r\mathbb S^1$ on the $k+1$-st factor of
$\C^m_k$.

Let $f\in\cont_{ J_A^{(k)}}(\C^m_k,\as)_0\cap\U_1$, where $\U_1$
is a sufficiently $C^1$-small neighborhood of the identity in
$\cont_{c}(\C_k^m,\as)_0$, with $\mu_{0}(f)\leq\frac{1}{2}$.
 For $q\in \C^m_{k+1}$ we choose $p=(\xi_0,\xi,y)\in\C^m_k$ with $\pi_k(p)=q$ and
$\xi_k<-A^5$. Let  $\tau_k=\tau_{k,1}$ is the unit translation
along the $x_k$ axis (c.f. section 2). Then we choose $l\in\N$
such that $((\tau_kf)^l(p))_{k}>A^5$. We define
$\Theta_{A}^{(k)}(f):\C^m_{k+1}\rightarrow \C^m_{k+1}$ by
\begin{equation*}
  \Theta_{A}^{(k)}(f)(q)=\pi_k((\tau_kf)^l(p)).
\end{equation*}
 The definition is independent of the choice
of $l$ and $p$.

\begin{prop} Let $k=0,1\ld n$. Possibly shrinking $\U_1$, the mapping
\[\Theta_A^{(k)}:\cont_{J^{(k)}_A}(\C_k^m,\as)_0\cap \U_1\r\cont_{J^{(k+1)}_A}(\C^m_{k+1},\ac)_0\]
 satisfies the following conditions:
\begin{enumerate}
\item $\Theta_A^{(k)}$ is continuous and it preserves the
identity.\item
$\Theta_A^{(k)}\big(\cont_{J^{(k)}_A}(\C_k^m,\as)_0^{(k)}\big)\s
\cont_{J^{(k+1)}_A}(\C^m_{k+1},\ac)_0^{(k)}$, c.f.(7.1).
 \item There exist   constants $\beta$, $K$, and
admissible polynomials $F_{r,A}$ for all $r\geq 1$ such that
 \begin{equation*}
\mu_r^*(\Theta_A^{(k)}(g))\leq \ab
K^r\mu^*_r(g)+F_{r,A}(M^*_{r-1}(g)),
\end{equation*}
for any $g\in\dom(\Theta_A^{(k)})$. Moreover, we may have
$F_{1,A}=0$.
\end{enumerate}
\end{prop}

\begin{proof} (1) and (2) are obvious.
A standard proof for (3) follows by virtue of Lemma 3.6 and Remark
3.7.

\end{proof}

In order to define the rolling-up operator $\Psi_A$ first we
introduce
\[\Xi_A^{(k)}:\cont_{J_A^{(k+1)}}(\C^m_{k+1},\as)_0\cap\U_1
\r\cont_{K_A^{(k)}}(\C^m_{k},\as)_0,\]
 where $k=0,1\ld n$ and $\U_1$ is a $C^1$ neighborhood of id in
 $\cont_c(\C^m_{k+1},\as)_0$.

Let $\psi:\mathbb S^1\r[0,1]$ be a smooth function such that
$\psi=1$ in a neighborhood of
$\left[-\frac{1}{8},\frac{1}{8}\right]$ and $\psi=0$ on
$\left[\frac{3}{8},\frac{5}{8}\right]$. Abusing the notation, let
$\psi:\C^m_{k+1}\r[0,1]$ such that
$\psi(\xi_0,\xi,y)=\psi(\xi_k)$. For
$g\in\cont_{J_A^{(k+1)}}(\C^m_{k+1},\as)_0\cap\U_1$   we define
\[g^{\psi}=\Phi_A^{-1}(\psi\Phi_A(g))=\Phi_A^{-1}(\psi
u_g),\]as in Definition 5.3. For short, set $\mathcal
E_{A,n,k}=[-A^5,A^5]^{n-k}\t[-2A,2A]^n$. Then
$g^{\psi}=g\quad\hbox{ on}\quad (\mathbb
S^1)^k\t\left[-\frac{1}{8},\frac{1}{8}\right]\t\mathcal E_{A,n,k}$
and $\supp(g^{\psi})\s (\mathbb
S^1)^k\t\left[-\frac{3}{8},\frac{3}{8}\right]\t\mathcal
E_{A,n,k}$, in view of Proposition 5.4.

 Let $g^{\psi}_1$ (resp. $g^{\psi}_2$) be the unique lift of $(g^{\psi})^{-1}g$
 (resp. $g^{\psi}$) to $\C_k^m$. Then $g^{\psi}_1$ and $g^{\psi}_2$
 are periodic contactomorphisms supported in $(\mathbb S^1)^k\t\R\t\mathcal E_{A,n,k}$.
 For small enough $\U_1$ there is $\varepsilon>0$ such that
 $g_1^{\psi}=g$ on $(\mathbb S^1)^k\t
 [\frac{1}{2}-\varepsilon,\frac{1}{2}+\varepsilon]\t\mathcal
 E_{A,n,k}$ and $g_2^{\psi}=g$ on $(\mathbb S^1)^k\t
 [1-\varepsilon,1+\varepsilon]\t\mathcal
 E_{A,n,k}$.

 Next we put
 $ E_k^-=\{(\xi_0,\xi,y)\in\C^m_k:-1\leq \xi_k\leq0\}$, and
  $E_k^+=\{(\xi_0,\xi,y)\in\C^m_k:\frac{1}{2}\leq \xi_k\leq\frac{3}{2}\}$,
and we define $\Xi_A^{(k)}(g)$ by the conditions
\begin{equation}
 \Xi_A^{(k)}(g)|_{E_k^-}=g^{\psi}_1|_{E_k^-},\quad \Xi^{(k)}_A(g)|_{E_k^+}=g^{\psi}_2|_{E_k^+}, \,
\end{equation}
and $\Xi^{(k)}_A(g)=\id$ on $\C^m_k\setminus(E_k^-\cup E_k^+)$.

\begin{prop} Taking $\U_1$ small enough, the mapping
$\Xi_A^{(k)}$ satisfies the following conditions:
\begin{enumerate}
\item $\Xi_A^{(k)}$ is continuous and it preserves the identity.
\item
$\Xi_A^{(k)}\big(\cont_{J_A^{(k+1)}}(\C^m_{k+1},\as)_0^{(k)}\big)
\s\cont_{K_A^{(k)}}(\C^m_{k},\as)_0^{(k)}$.

 \item  There are  constants  $\cpp$ , $\beta$ and
$K$,  and for any $r\geq 2$ there is a polynomial with no constant
term $P_{\phi,\psi,r}$ such that for any $g\in\dom(\Xi^{(k)}_A)$
one has
\begin{equation*} \mu^*_r(\Xi_A^{(k)}(g))\leq
 K^rC_{\phi,\psi}\mu_r^*(g)+\ab P_{\phi,\psi,r}(M^*_{r-1}(g)).
\end{equation*}
In particular, $ \mu^*_r(\Xi_A^{(k)}(g))\leq\ab
C_{\phi,\psi,r}\mu_r^*(g)$ whenever $\supp(g)\s E$ with $R_E\leq
2$. \item For any $g\in\dom(\Xi^{(k)}_A)$ one has
$\Theta^{(k)}_{A} \Xi^{(k)}_{A}(g)=g$.

\end{enumerate}
\end{prop}

\begin{proof}
The properties (1) and (4) can be deduced from the definition. To
check (2) we use Proposition 4.3. Finally, as
$\mu^*_r(\Xi^{(k)}_A(g))\leq
\max\{\mu^*_r(g^{\psi}_1),\mu^*_r(g^{\psi}_2)\}$,  (3) follows
from  Lemmas 5.5 and 3.6.
\end{proof}

It will  be useful to introduce   operators
\[\Theta^{(k)}:\cork\cap\U_1\r\cont_c(\C^m_{k+1},\as)_0,\quad
 k=0\ld n,\]
obtained by gluing-up $\Theta^{(k)}_A$,
$\Theta^{(k)}=\bigcup\Theta^{(k)}_A$. Now, let us return to the
"hat" operation defined by (7.2). For
$f\in\cont_{E_A^{(k)}}(\C^m_k,\as)_0^{(k)}\cap\U_1$ denote
$\hat\Theta_A^{(k)}(f)=\widehat{\Theta_A^{(k)}(f)}$. We set
$\hat\Theta^{(k)}=\bigcup\hat\Theta^{(k)}_A$ and we have operators
 \[\hat\Theta^{(k)}:\cont_c(\C_k^m,\as)_0^{(k)}\cap\U_1\r\cont_c(\C_{k+1}^m,\as)_0^{(k+1)},\quad
 k=0\ld n.\]
 Likewise
$\Xi^{(k)}=\bigcup\Xi^{(k)}_A$, that is we have
\[\Xi^{(k)}:\cont_c(\C^m_{k+1},\as)_0\cap\U_1\r\cork,\quad
 k=0\ld n.\]

\begin{lem}
If $f,g\in\dom(\Theta^{(k)})$ and
$\Theta^{(k)}(f)=\Theta^{(k)}(g)$ then $[f]=[g]$ in $H_1(\cork)$.
\end{lem}
\begin{proof}
Let  us define a contactomorphism $\Lambda_k=\Lambda_k(f,g)$ by
\begin{equation} \Lambda_k(p)
=(\tau_k g)^l(\tau_kf)^{-l}(p),
\end{equation}
 where $p\in\C^m_k$,
$\tau_k=\tau_{k,1}$ is the translation, and $l$ is a positive
integer so large that $[(\tau_k f)^{-l}(p)]_{k}<-A^5$. Clearly,
$\Lambda_k$ does not depend on  $l$, and
$\Lambda_k\in\cont_c(\C^m_k,\as)_0$ in view of the definition of
$\Theta^{(k)}$ and the assumption. From (8.3) we have $ \Lambda_k
\tau_kf\Lambda_k^{-1}=\tau_kg$ and, consequently, $[f]=[g]$.
\end{proof}

\begin{lem} Let $k=0,1\ld n$.
\begin{enumerate}

\item If $\Theta^{(k)}(f_i)=g_i$, $i=1\ld l$, then there are $\bar
f_i\in\cork$ such that $\Theta^{(k)}(f)=g_1\ldots g_l$, where
$f=\bar f_1\ldots \bar f_l$, and $[\bar f_i]=[f_i]$ in
$H_1(\cork)$ for all $i$. Moreover, we can have $\bar f_1=f_1$.

\item If $g_1,g_2, g_1g_2\in\dom(\Xi^{(k)})$ then
$[\Xi^{(k)}(g_1g_2)]=[\Xi^{(k)}(g_1)\Xi^{(k)}(g_2)]$ in the group
$H_1(\cork)$.

\item If $g\in\cont_c(\C^m_{k+1},\as)_0$ with $[g]=e$ in
$H_1(\cont_c(\C^m_{k+1},\as)_0)$ then there is $f\in\cork$ such
that $\Theta^{(k)}(f)=g$ and $[f]=e$ in $H_1(\cork)$.

\end{enumerate}
\end{lem}
\begin{proof}
(1) We may shift supports of $f_i$ by the  translations
$\tau_{k,t}$ to obtain $\bar f_i$ such that the family $\{\bar
f_i\}$ has pairwise disjoint supports. Clearly $[\bar f_i]=[f_i]$.
 Moreover , by definition of $\Theta^{(k)}$ we can arrange $\bar
f_i$ in the way that $\Theta_A^{(k)}(f)=g_1\ldots g_l$ for $f=\bar
f_1\ldots \bar f_l$ and $\bar f_1=f_1$.

(2) Put $f_i=\Xi^{(k)}(g_i)$, $i=1,2$. In view of (1) there is
$f\in\dom(\Theta^{(k)})$ such that $[f]=[f_1f_2]$ and
$\Theta^{(k)}(f)=g_1g_2$. By Proposition 8.2(4),
$\Theta^{(k)}\Xi^{(k)}(g_1g_2)=g_1g_2=\Theta^{(k)}(f)$. Therefore,
from Lemma 8.3
\[[\Xi^{(k)}(g_1g_2)]=[f]=[f_1f_2]=[\Xi^{(k)}(g_1)\Xi^{(k)}(g_2)].\]

(3) First we define an operator
\[\bar\Xi^{(k)}:\cont_c(\C^m_{k+1},\as)_0\cap\U_1\r\cork,\quad
 k=0\ld n,\]
with $\dom(\bar\Xi^{(k)})=\dom(\Xi^{(k)})$ such that for any
$g\in\dom(\bar\Xi^{(k)})$ we have
$[\bar\Xi^{(k)}(g)]=[\Xi^{(k)}(g)^{-1}]$ and
$\Theta^{(k)}\bar\Xi^{(k)}(g)=g^{-1}$.

Namely, let us return to the definition of
$\Xi^{(k)}(g)=\Xi^{(k)}_A(g)$. We have the decomposition
$g=g^{\psi}_2g^{\psi}_1$, where $\psi$ is a suitable smooth
function. Now, we define $\bar\Xi^{(k)}(g)$ by changing (8.2) as
follows
\begin{equation*}
 \bar\Xi_A^{(k)}(g)|_{E_k^-}=\tau_{k,\frac{3}{2}}^{-1}\ci(g^{\psi}_2)^{-1}|_{E_k^+}\ci\tau_{k,\frac{3}{2}},
 \quad \bar\Xi^{(k)}_A(g)|_{E_k^+}=\tau_{k,\frac{3}{2}}\ci (g^{\psi}_1)^{-1}|_{E_k^-}\ci\tau_{k,\frac{3}{2}}^{-1},
\end{equation*}
and $\bar\Xi^{(k)}_A(g)=\id$ on $\C^m_k\setminus(E_k^-\cup
E_k^+)$.

By assumption there are $h_j\in\cont_c(\C^m_{k+1},\as)_0$, $j=1\ld
2l$, such that
\[g=[h_1,h_2]\ldots[h_{2l-1},h_{2l}].\] For all $j$ we may write
a decomposition $h_j=h_{j,1}\cdots h_{j,l(j)}$, where the factors
are $C^1$ small.

Put $ f_{j,s}=\Xi^{(k)}(h_{j,s})$   and $
f_{j,s}^*=\bar\Xi^{(k)}(h_{j,s})$, $j=1\ld 2l, s=1\ld l(j)$. Let
us define $f_j=\bar f_{j,1}\cdots\bar f_{j,l(j)}$ and $f_j^*=\bar
f_{j,l(j)}^*\cdots\bar f_{j,1}^* $ as in the proof of (1). In
particular, $\Theta^{(k)}(f_j)=h_j$,
$\Theta^{(k)}(f_j^*)=h_j^{-1}$, and $[f_j^*]=[f_j^{-1}]$ for
$j=1\ld 2l$. Therefore, in view of (1), the claim follows.
\end{proof}

Next we introduce the auxiliary rolling-up operators.
\begin{prop} Let $r\geq 2$ and let $k=0,1\ld n$. There exist a $C^r$ neighborhood $\U_1=\U_{\phi,\psi, r, A,k}$
of the identity in $\cork$  and  a mapping
$\Psi_A^{(k)}=\Psi_{\phi,\psi,r,A,k}$ such that
\[\Psi_{A}^{(k)}:\cont_{J_A^{(k)}}(\C_k^m,\as)_0^{(k)}\cap \U_1
\r\cont_{K^{(k)}_{A}}(\C_k^m,\as)_0^{(k)},\]  c.f. (7.1), which
satisfies the following conditions:
\begin{enumerate}
\item $\Psi^{(k)}_{A}$ is  continuous and $\Psi^{(k)}_A(\id)=\id$.
\item  There are constants $C_{\phi}$,  $\beta$ and $K$,  and for
any $\rho\geq 2$  polynomials with no constant term
$P_{\phi,\psi,\rho}$ such that for any $g\in\dom(\Psi^{(k)}_A)$
one has
\begin{equation*}
\mu_{\rho}^*(\Psi^{(k)}_A(g))\leq \ab
K^{\rho}C_{\phi}\mu_{\rho}^*(g)+\ab
P_{\phi,\psi,\rho}(M^*_{\rho-1}(g)).
\end{equation*}

\item  There is a constant $C_{\phi,\psi,r}$ such that for all
$g\in\dom(\Psi^{(k)}_A)$ one has \[ \mu^*_r(\Psi^{(k)}_A(g))\leq
\ab C_{\phi,\psi,r}\mu^*_r(g).\]

\item For any $g\in\dom(\Psi^{(k)}_A)$ we have $[
\Psi^{(k)}_{A}(g)\cdot\Xi^{(k)}\hat\Theta^{(k)}(g)]=[g]$ in the
group $H_1(\cork)$.

\end{enumerate}
\end{prop}

\begin{proof}

Let $g\in\cont_{J^{(k)}_A}(\C_k^m,\as)^{(k)}_0\cap \U_1$. Define
\[
\Psi_A^{(k)}(g):=\Xi_A^{(k)}\big(\Theta_A^{(k)}(g)\cdot\hat\Theta_A^{(k)}(g)^{-1}\big)
=\Xi^{(k)}\big(\Theta^{(k)}(g)\cdot\hat\Theta^{(k)}(g)^{-1}\big),\]
c.f.  (7.2). By virtue of Propositions 8.1 and 8.2 the definition
is correct and (1)  holds true. To show (2) and (3) denote
$h=\Theta^{(k)}(g)\cdot\hat\Theta^{(k)}(g)^{-1}$. Then
$u_h=w_{\Theta^{(k)}(g)}$, c.f. (7.3). According to (5.1), and
Propositions 4.6 and 7.2(5)  we have
\begin{equation}
 \mu^*_{\rho}(h^{\psi})\leq
  \ab K^{\rho} C_{\phi}
\mu^*_{\rho}(\Theta_A^{(k)}(g))+\ab
P_{\phi,\psi,\rho}(M_{\rho-1}^*(\Theta^{(k)}_A(g))).
\end{equation}
Next, by (5.1) and Corollary 7.3
\begin{equation}
\forall 0\leq s\leq r+1,\quad \|D^s(\psi u_h)\|\leq\ab
C_{\phi,\psi,r} \mu^*_r(\Theta^{(k)}_A(g)).
\end{equation}
Now, (2) follows from (8.4), Lemma 3.6 and Propositions 8.1 and
8.2. On the other hand, by (8.5) with Propositions 4.6 and 8.1 we
obtain
\[ \mu^*_r(h^{\psi})\leq\ab C_{\phi,\psi,r}
\mu^*_r(g)+F_{r,A}(M^*_{r-1}(g)).\] In view of the definition of
$\Xi_A^{(k)}$, Lemma 3.6 and Definition 3.2, the claim (3) follows
by shrinking possibly $\U_1$.

(4)  We have by Lemma 8.4(2)
\begin{align}
\begin{split}[\Psi^{(k)}_A(g)\cdot\Xi^{(k)}\hat\Theta^{(k)}(g)]&=
[\Xi^{(k)}\big(\Theta^{(k)}(g)\cdot\hat\Theta^{(k)}(g)^{-1}\big)\cdot\Xi^{(k)}\hat\Theta^{(k)}(g)]\\
&=[\Xi^{(k)}\Theta^{(k)}(g)].\end{split}
\end{align}
Notice that in view of Proposition 8.2(4) we get
$\Theta^{(k)}\Xi^{(k)}\Theta^{(k)}(g)=\Theta^{(k)}(g)$. It follows
from Lemma 8.3 that $[\Xi^{(k)}\Theta^{(k)}(g)]=[g]$. Combining
this with (8.6), the claim follows.

\end{proof}

From now on we set for $k=0,1\ld n$
\[\tilde\Theta^{(k)}=\Theta^{(k)}\ci\cdots\ci\Theta^{(0)},\quad
\Theta_*^{(k)}=\hat\Theta^{(k)}\ci\cdots\ci\hat\Theta^{(0)},\quad
\tilde \Xi^{(k)}=\Xi^{(0)}\ci\cdots\ci\Xi^{(k)}.\] Notice that the
image of  $\Theta_*^{(k)}$ is in
$\cont_c(\C^m_{k+1},\as)_0^{(k+1)}$.

In the proof of Theorem 1.1 the following fact is crucial.

\begin{lem}
Suppose $\U_1$ is a sufficiently small $C^1$ neighborhood of the
identity in $\cont_c(\R^m,\as)_0$.  Then for all $f\in\U_1$:
\begin{enumerate}
\item If $\tilde\Theta^{(n)}(f)=\tilde f$ then
$[f]=[\tilde\Xi^{(n)}(\tilde f)]$ in $H_1(\corr)$.

 \item $[\tilde\Xi^{(n)}\Theta^{(n)}_*(f)]=e$ in
$H_1(\corr)$.
\end{enumerate}
\end{lem}
\begin{proof}
(1) In view of Proposition 8.2(4) and Lemma 8.3 one has
$[\tilde\Theta^{(n-1)}(f)]=[\Xi^{(n)}(\tilde f)]$. Hence there is
$h_n$ in the commutator subgroup of $\cont_c(\C^m_n,\as)_0$ such
that $\tilde\Theta^{(n-1)}(f)h_n=\Xi^{(n)}(\tilde f)$. By the
above argument and Lemmas 8.3, 8.4(3) and 8.4(1), there is
$h_{n-1}$ in the commutator subgroup of
$\cont_c(\C^m_{n-1},\as)_0$ such that
$\tilde\Theta^{(n-2)}(f)h_{n-1}=\Xi^{(n-1)}\Xi^{(n)}(\tilde f)$.
Continuing this procedure we obtain the claim.

(2) For $f\in\U_1$ put
\begin{equation*}
f^*=\Theta_*^{(n)}(f)\quad\hbox{and}\quad g=
\tilde\Xi^{(n)}\Theta^{(n)}_*(f)=\tilde\Xi^{(n)}(f^*).
\end{equation*}
Notice that in view of Proposition 2.2,
$f^*(\xi_0,\xi,y)=(\xi_0+f^*_0(y), \xi+f^*_1(y),y)$. It follows
from the definition of $\tilde\Xi^{(n)}$ that
$g(\xi_0,\xi,y)=(\xi_0+f^*_0(y), \xi+f^*_1(y),y)$ if
$(\xi_0,\xi)\in([-\frac{1}{2}-\varepsilon,-\frac{1}{2}+\varepsilon]\cup[1-\varepsilon,1+\varepsilon])^{n+1}$
for some $\varepsilon>0$. Furthermore,
$\supp(g)\s([-1,0]\cup[\frac{1}{2},\frac{3}{2}])^{n+1}\t[-2A,2A]^n$
and, due to Proposition 8.2(4), $\tilde\Theta^{(n)}(g)=f^*$. We
have to show that $[g]=e$.

Let us define $g_2\in\corr$ such that
\begin{equation}
[g]=[g_2]=[g^{2^{n+2}}] \end{equation} in $H_1(\corr)$. In the
definition we will use the contactomorphisms $\eta_2,
\tau_{i,t}\in\corr$, $i=0\ld n$, defined in section 2.

First let $ h=\eta_2^{-1}g\eta_2$. Then
$\supp(h)\s[-\frac{1}{2},0]\cup[\frac{1}{4},\frac{3}{4}]\t\mathcal
I_{n,A}$, where $\mathcal
I_{n,A}=([-\frac{1}{2},0]\cup[\frac{1}{4},\frac{3}{4}])^{n}\t[-2A,2A]^n$.
Let us denote
\begin{equation*}
f^*_{\frac{1}{2}}(\xi_0,\xi,y)=\left(\xi_0+\frac{1}{2}f^*_0(y),
\xi+\frac{1}{2}f^*_1(y),y\right).
\end{equation*}
To simplify notation, put $\mathcal
J_{l,\varepsilon}=([-\frac{1}{4}-\varepsilon,-\frac{1}{4}+\varepsilon]\cup
[\frac{1}{2}-\varepsilon,\frac{1}{2}+\varepsilon])^{l}\t\R^n$.
There is $\varepsilon>0$ such that for $ (\xi_0,\xi,y)\in\mathcal
J_{n+1,\varepsilon}$ one has
$h(\xi_0,\xi,y)=f^*_{\frac{1}{2}}(\xi_0,\xi,y)$.

We can write $h=\bar h_0\hat h_0$, where  $\bar h_0=h$ on
$[-\frac{1}{2},0]\t\R^{2n}$, $\bar h_0=\id$ off
$[-\frac{1}{2},0]\t\R^{2n}$, $\hat h_0=h$ on
$[\frac{1}{4},\frac{3}{4}]\t\R^{2n}$, $\hat h_0=\id$ off
$[\frac{1}{4},\frac{3}{4}]\t\R^{2n}$. Put $h_0=\hat
h_0\tau_{0,\frac{1}{2}}\bar h_0\tau_{0,\frac{1}{2}}^{-1}$. Clearly
$[h_0]=[h]$. Observe that $ h_0=f^*_{\frac{1}{2}}$ on
$[\frac{1}{4}-\varepsilon,\frac{1}{2}+\varepsilon]\t\mathcal
J_{n,\varepsilon}$ and $\supp(h_0)\s[0,\frac{3}{4}]\t\mathcal
I_{n,A}$.

In view of the equalities $\eta_2^{-1}
f^*\eta_2=f^*_{\frac{1}{2}}$ (here $f^*$ is regarded as an element
of $\cont(\R^m,\as)$) and $\eta_2^{-1}
\tau_{0,1}\eta_2=\tau_{0,\frac{1}{2}}$, we have
$h_0\tau_{0,\frac{1}{2}}h_0=f^*_{\frac{1}{2}}$ on
$[0,\frac{1}{4}]\t\mathcal J_{n,\varepsilon}$ and, moreover, by
the definition of $\Xi^{(n)}$ we get
$h_0\tau_{0,\frac{1}{2}}h_0=\eta_2^{-1}
\tilde\Xi^{(n-1)}(f^*)\eta_2$ on $[0,\frac{1}{4}]\t\R^{2n}$. Here
$\tilde\Xi^{(n-1)}(f^*)\in\cont(\C^m_1,\as)$ is viewed as an
element of $\cont(\R^m,\as)$ with period 1 w.r.t. $x_0$, so that
$\eta_2^{-1} \tilde\Xi^{(n-1)}(f^*)\eta_2$ is well-defined and can
be also regarded as an element of $\cont(\C^m_1,\as)$ with period
$\frac{1}{2}$ w.r.t. $\xi_0$.

Next we define $k_0=h_0\tau_{0,\frac{1}{2}}
h_0\tau_{0,\frac{1}{2}}^{-1}$. We have
$\supp(k_0)\s[0,\frac{5}{4}]\t\mathcal I_{n,A}$ and
$k_0=f^*_{\frac{1}{2}}$ on
$[\frac{1}{4}-\varepsilon,1+\varepsilon]\t\mathcal
J_{n,\varepsilon}$. Analogously as above,
$k_0\tau_{0,1}k_0=f^*_{\frac{1}{2}}$ on $[0,\frac{1}{4}]\t\mathcal
J_{n,\varepsilon}$ and $k_0\tau_{0,1}k_0=\eta_2^{-1}
\tilde\Xi^{(n-1)}(f^*)\eta_2$ on $[0,\frac{1}{4}]\t\R^{2n}$.

It follows from the definition of $\Theta^{(0)}$ that
 $\Theta^{(0)}(k_0)=f^*_{\frac{1}{2}}$ on $\mathbb
S^1\t\mathcal J_{n,\varepsilon}$ and
$\Theta^{(0)}(k_0)=\eta_2^{-1}\tilde\Xi^{(n-1)}(f^*)\eta_2$ on
$\C^m_1$. One has also that $[k_0]=[h_0^2]=[h^2]=[g^2]$.

Next, starting with $k_0$, we define $\bar h_1, \hat h_1, h_1$ and
$k_1$ analogously as before, but now with respect to the variable
$\xi_1$. It follows that
$\tilde\Theta^{(1)}(k_1)=f^*_{\frac{1}{2}}$ on $(\mathbb
S^1)^2\t\mathcal J_{n-1,\varepsilon}$, and
$\tilde\Theta^{(1)}(k_1)=\eta_2^{-1}\tilde\Xi^{(n-2)}(f^*)\eta_2$
on $\C^m_2$. Moreover, $[k_1]=[k_0^2]=[h^4]=[g^4]$.

Continuing this procedure  we obtain $h_2,k_2\ld h_n,k_n\in\corr$
such that
$[k_n]=[k_{n-1}^{2}]=[k_{n-2}^4]=\ldots=[k_0^{2^{n}}]=[g^{2^{n+1}}]$.
Moreover,  we have that
$\tilde\Theta^{(n)}(k_n)=f^*_{\frac{1}{2}}$ on $\C^m_{n+1}$.

Thus, in order to define $g_2$ satisfying
$\tilde\Theta^{(n)}(g_2)=f^*$ we have to double $k_n$ and we set
$g_2=\tau k_n\tau^{-1}k_n$, where $\tau$ is a suitable
translation, as in Lemma 8.4(1). It follows that $[g_2]=[k_n^2]$
and, in view of (1) of the present lemma, the equalities (8.7)
hold.

Observe that the above procedure may be repeated for any integer
$a>2$ by making use of $\eta_a$ and suitable translations
$\tau_{i,t}$. As a result there exists $ g_a\in\corr$ such that
$\tilde\Theta^{(n)}( g_a)=f^*$ and $[g^{a^{n+2}}]=[g_a]$.
Moreover, by (1) we have $[g_a]=[g]$.

Let $l_0>0$ be the least positive integer such that $[g^{l_0}]=e$.
Then for any integers $a,b>0$ the number $a^{n+2}-b^{n+2}$ is
divided by $l_0$. If $l_0>1$ then $l_0$ divides $l_0^{n+2}-1$, a
contradiction. Thus $l_0=1$, as required.
\end{proof}

\begin{prop} Let $r\geq 2$. If $\U_1=\U_{\phi,\psi,r,A}$ is
a small $C^r$ neighborhood of the identity in $\corr$, there is  a
mapping $\Psi_A=\Psi_{\phi,\psi,r,A}$, called the rolling-up
operator,
\[\Psi_{A}:\cont_{J_A}(\R^m,\as)_0\cap \U_1
\r\cont_{K_{A}}(\R^m,\as)_0,\] which satisfies the following
conditions:
\begin{enumerate}
\item $\Psi_{A}$ is  continuous and $\Psi_A(\id)=\id$.

\item  There are  constants $C_{\phi}$,  $\beta$ and $K$,  and for
any $\rho\geq 2$ there is polynomials with no constant term
$P_{\phi,\psi,\rho}$ such that for any $g\in\dom(\Psi_A)$
\begin{equation*} \mu_{\rho}^*(\Psi_A(g))\leq
\ab K^{\rho}C_{\phi}\mu_{\rho}^*(g)+\ab
P_{\phi,\psi,\rho}(M^{\l}_{\rho-1}(g)).
\end{equation*}

\item There are constants $\beta$ and $C_{\phi,\psi,r}$, and an
admissible polynomial $F_{r,A}$ such that for any
$g\in\dom(\Psi_A)$
\[\mu^*_r\big(\Psi_A(g)\big)\leq
\ab C_{\phi,\psi,r}\mu^*_r(g)+F_{r,A}(M^*_{r-1}(g)).\]

\item For any $g\in\dom(\Psi_A)$ one has $[\Psi_{A}(g)]=[g]$ in
$H_1(\corr)$.

\end{enumerate}
\end{prop}
\begin{proof}
Let $g\in \cont_{J_A}(\R^m,\as)_0\cap \U_1$. Define
$\Psi_A(g)=g_0g_1\ldots g_n$, where $g_0=\Psi_A^{(0)}(g)$ and, for
$k=1\ld n$,
\[ g_k=\tilde\Xi^{(k-1)}\Psi_A^{(k)}\Theta_*^{(k-1)}(g).\]

In order to show (2) and (3), our first observation is that it
suffices to have for $k=0,1\ld n$
\begin{equation}
\mu_{\rho}^*(g_k)\leq \ab K^{\rho}C_{\phi}\mu_{\rho}^*(g)+\ab
P_{\phi,\psi,\rho}(M^*_{\rho-1}(g)), \end{equation} for all
$\rho\geq 2$, and \begin{equation} \mu^*_r(g_k)\leq
A^{\beta}C_{\phi,\psi,r}\mu^*_r(g),
\end{equation}
 and to apply Lemma 3.6(2). For $k=0$ it is just Proposition 8.5.

For $k=1\ld n$, in view of Propositions 8.2 and 8.5 we get
\begin{equation}
\mu_{\rho}^*(g_k)\leq\ab K^{\rho}
C_{\phi}\mu^*_{\rho}(\Theta^{(k-1)}_*(g))+\ab
P_{\phi,\psi,\rho}(M^*_{\rho-1}(\Theta^{(k-1)}_*(g))).
\end{equation}
On the other hand, by Propositions 7.2(4) and 8.1(3) we have
\begin{align*}
\mu^*_{\rho}(\Theta^{(k-1)}_*(g))&\leq
K^{\rho}C_{\phi} \mu^*_{\rho}(\Theta^{(k-1)}\Theta^{(k-2)}_*(g))+P_{\phi,\rho}(\Theta^{(k-1)}\Theta^{(k-2)}_*(g))\\
&\leq A^{\beta}K_1^{\rho}
C'_{\phi}\mu^*_{\rho}(\Theta^{(k-2)}_*(g))+P'_{\phi,\rho}(M^*_{\rho-1}(\Theta^{(k-2)}_*(g)))\\
&\quad\cdots\\
&\leq  A^{\beta'}K_2^{\rho}
 C''_{\phi}\mu^*_r(g)+ P''_{\phi,\rho}(M^*_{\rho-1}(g)).
\end{align*}
Combining this with (8.10) we obtain (8.8). In order to show (8.9)
for $k=1\ld n$ we proceed analogously, using  Propositions 8.2,
3.5, 8.5 and 8.1, and Corollary 7.3, and possibly changing
constants and shrinking $\U_1$
\begin{align*}
\mu^*_r(g_k) &\leq
\ab C_{\phi,\psi,r}\mu^*_r(\Psi^{(k)}\Theta^{(k-1)}_*(g))\\
&\leq
\ab C_{\phi,\psi,r}\mu^*_r(\Theta^{(k-1)}_*(g))\\
&\leq
\ab C_{\phi,\psi,r}\mu^*_r(\Theta^{(k-1)}\Theta^{(k-2)}_*(g))\\
&\leq
\ab C_{\phi,\psi,r}\mu^*_r(\Theta^{(k-2)}_*(g))\\
&\quad\cdots\\
&\leq
\ab C_{\phi,\psi,r}\mu^*_r(\Theta^{(0)}(g))\\
&\leq  \ab
 C_{\phi,\psi,r}\mu^*_r(g).
\end{align*}

(4) By Lemmas 8.6(2) and 8.4(2), and  Proposition 8.5(4) , we have
\begin{align*}
[\Psi_A(g)]&=[g_0g_1\cdots g_n]\\
&=[ g_0g_1\cdots g_n\cdot\tilde\Xi^{(n)}\Theta^{(n)}_*(g)]\\
&=[ g_0g_{1}\cdots
g_{n-1}\cdot\tilde\Xi^{(n-1)}\Psi_A^{(n)}\Theta^{(n-1)}_*(g)
\cdot\tilde\Xi^{(n-1)}\Xi^{(n)}\hat\Theta^{(n)}\Theta^{(n-1)}_*(g)
]\\
&=[g_0g_{1}\cdots
g_{n-1}\cdot\tilde\Xi^{(n-1)}\big(\Psi_A^{(n)}\Theta^{(n-1)}_*(g)\cdot\Xi^{(n)}\hat\Theta^{(n)}\Theta^{(n-1)}_*(g)
\big)]\\
&=[g_0g_{1}\cdots
g_{n-1}\cdot\tilde\Xi^{(n-1)}\Theta^{(n-1)}_*(g)]\\
&\quad\cdots \\
&=[g_0\cdot\Xi^{(0)}\Theta^{(0)}_*(g)]\\
&=[\Psi_A^{(0)}(g)\cdot\Xi^{(0)}\hat\Theta^{(0)}(g)]=[g].
\end{align*}
\end{proof}

\begin{rem}
It is easy to check that the proof of Lemma 8.6(2) and,
consequently, of Proposition 8.7(4) fails in the case
$\diff^r_c(\R^m)_0$, since Proposition 2.2 is not true for
diffeomorphisms. Thus the proof of Theorem 1.1 is not valid for
$\diff^r_c(\R^m)_0$.
\end{rem}

\section{Proof of Theorem 1.1}

Let $A$ be a large positive  integer  which will be fixed later
on, and let $I_A$, $J_A$ and $K_A$ be the intervals in $\R^m$
given by (6.1), (6.3) and (8.1), resp.
 Let us define
\begin{equation*}
\L=\{u\in\cc_{I_A}(\rr)\,:\,\|D^{r+1}u\|\leq\epsilon_r,\forall
r\geq r_0\},
\end{equation*}
where $r_0$ (large), $\epsilon_{r_0}$ ( small),  and $\epsilon_r$
for $r>r_0$ (large) will be fixed in due course.

Observe that $\L$  is a convex and compact subset of a locally
convex space. Consequently, in view of Schauder-Tychonoff's
theorem every continuous map $\vartheta:\L\rightarrow \L$ has a
fixed point.

Let $f_0\in\corr$. We have to show that $f_0$ belongs to the
commutator subgroup of $\corr$. According to Lemma 5.2 we may
assume that $\supp(f_0)\s I_A$. Furthermore, since $\corr$ is a
topological group, we may have $\mu^{\l}_{r_0}(f_0)$  arbitrarily
small.

Now we will define a continuous operator $\vartheta: \L\r \L$ in
the following ten steps:

\begin{enumerate}
\item For any $u\in \L$ take $f\in\cont_{I_A}(\rr,\as)_0$ such
that $\Phi_A(f)=u_f=u$.

\item Compose $f$ with $f_0$.

 \item Use
a fragmentation  of the second kind for $g=ff_0$ (Proposition
5.7). We have a decomposition $g=g_1\cdots g_{a_n}$, where
$a_n=(4A+1)^{n}$, and each $g_{\kappa}$ is supported in some
interval
\[([-2,2]^{n+1}\t[k_1-1,k_1+1]\t\cdots\t[k_{n}-1,k_{n}+1])\cap
I_{A},\] with  integers $k_i$  such that $|k_i|\leq 2A$, $i=1\ld
n$.

\item Use the operation of shifting supports of contactomorphisms
described in section 6. For any $\kappa=1\ld \ss$ define
\[
\tilde
g_{\kappa}=\sigma_{n,t_{n}}\sigma_{n-1,t_{n-1}}\cdots\sigma_{1,t_{1}}
g_{\kappa}\sigma_{1,t_{1}}^{-1}\cdots
\sigma_{n-1,t_{n-1}}^{-1}\sigma_{n,t_{n}}^{-1},\] for suitable
$(t_{1}\ld t_{n})\in\R^n$ depending on $\kappa$ in such a way that
$\supp(\tilde g_{\kappa})\s [-A^2,A^2]\t[-2,2]^{2n}$ for all
$\kappa$. Here we assume that $|t_i|\leq 2A$, $i=1\ld n$, and
$A>5n$.

 \item
For any $\kappa=1\ld \ss$ define $ h_{\kappa}=\eta_A\chi_{A}\tilde
g_{\kappa}\chi_{A}^{-1}\eta_A^{-1}$. It follows that
  $\supp( h_{\kappa})\s J_{A}$.

 \item Use the rolling-up operator
$\Psi_{A}$ described in Proposition 8.7, and define $\bar
h_{\kappa}=\Psi_{ A}( h_{\kappa})$. Observe that $\supp(\bar
h_{\kappa})\s K_{A}$.

\item Make a fragmentation of the second kind in $K_A$  in the
$x_i$-directions, $i=1\ld n$, c.f. Proposition 5.7. We write for
$\bar a_n=\ss^5$
\[\bar h_{\kappa}=\prod_{\iota=1}^{\bar a_n}\bar h_{\kappa\iota}.\]

 \item Use the operation of shifting
supports of contactomorphisms in the $x_i$-directions by means of
the translations $\tau_i$ , $i=1\ld n$ (c.f. section 2). For any
$\kappa$ and $\iota$ define $\tilde h_{\kappa\iota}$ instead of
$\bar h_{\kappa\iota}$ with $\supp(\tilde h_{\kappa\iota})\s I_A$.
All the norms of $\tilde h_{\kappa\iota}$ are the same as the
norms of $\bar h_{\kappa\iota}$ as we used  translations.

\item Take the product
$h=\prod_{\kappa=1}^{\ss}\prod_{\iota=1}^{\bar a_n}\tilde
h_{\kappa\iota}$.

\item Take $u_h=\Phi_A(h)$.
\end{enumerate}

 Then we put $\vartheta(u)=u_h$. In view of
the description of particular steps of the construction,
$\vartheta$ is continuous. It remains to show that for a suitable
choice of $r_0$, $A$, and $\epsilon_r$ for $r\geq r_0$, the
operator $\vartheta$ takes $\L$ into itself.

In fact, suppose that $u=u_f\in \L$ is a fixed point of
$\vartheta$, i.e, $u_h=u_f$. Then $h=f$ and we have in
$H_1(\corr)$
\begin{align*} [ff_0]&=[g]=[g_1\cdots g_{\ss}]=[ g_1]\cdots
[g_{\ss}]=[\tilde g_1]\cdots
[\tilde g_{\ss}]\\
&=[h_1]\cdots [h_{\ss}]= [\bar h_1]\cdots [\bar h_{\ss}]= [\bar
h_{11}]\cdots [\bar h_{\ss\bar a_n}]\\
&= [\tilde h_{11}]\cdots [\tilde h_{\ss\bar a_n}]=[\tilde
h_{11}\cdots \tilde h_{\ss\bar a_n}]=[h]=[f],\end{align*} and
therefore $[f_0]=e$. This means that $f_0$ is a product of
commutators.

Now we wish to define $r_0$, $A$ and $\epsilon_r$ for $r\geq r_0$.
This will be done in view of the properties of the consecutive
operations in the construction of $\vartheta$.

Suppose $r_0\geq 2$. In view of Propositions 4.6, 3.5, 5.7, 6.1
and 8.7, and Lemma 3.6 it follows the existence of a $C^2$
neighborhood $\V_2=\V_{\phi,\psi,r_0,A}$ of zero in
$\cc_{I_A}(\rr)$, of  constants $C_{\phi,\psi,r_0}$ and
$\beta=\beta(m)>0$, and of admissible polynomials $F^i_{r_0,A}$,
$i=1,3$, and $F^2_{\phi,\psi,r_0,A}$,  such that for a
sufficiently small $\epsilon_{r_0}$   we have
\begin{align}
\begin{split}
\|D^{r_0+1}u_h\|&\leq A^{\beta-r_0}
C_{\phi,\psi,r_0}\|D^{r_0+1}u\|+F^1_{r_0,A}(\mu^*_{r_0}(f))\\
&+F^2_{\phi,\psi,r_0,A}\big(\sup_{\kappa}\mu^*_{r_0}(
h_{\kappa})\big)+F^3_{r_0,A}\big(\sup_{\kappa,\iota}\mu^*_{r_0}(\tilde
h_{\kappa\iota})\big),
\end{split}
\end{align}
for all $u\in\V_{2}$ with $\|D^{r_0+1}u\|\leq\epsilon_{r_0}$. Here
we assume that $\mu_{r_0}^*(f_0)$ is small enough. We assume as
well that $\sup_{\kappa,\iota}\mu_i^*(\tilde
h_{\kappa\iota})<(A^{20n}r_0)^{-1}$, where $i=0,1$, by choosing
$\epsilon_{r_0}$ sufficiently small. Then we have
\begin{equation}\big(\big(1+\sup_{\kappa,\iota}\mu_0^*(\tilde h_{\kappa\iota})\big)
    \big(1+\sup_{\kappa,\iota}\mu_1^*(\tilde
    h_{\kappa\iota})\big)\big)^{A^{20n}r_0}<6,\end{equation}
and we may apply Lemma 3.6(2) in order to obtain (9.1).

 Fix $r_0>\beta$ and
choose $A$ so large that
$A^{\beta-r_0}C_{\phi,\psi,r_0}<\frac{1}{4}$. It follows from
Definition 3.2  that, possibly taking $\epsilon_{r_0}$ smaller, we
have $F^1_{r_0,A}(\mu^*_{r_0}(f))<\frac{\epsilon_{r_0}}{4}$,
$F^2_{\phi,\psi,r_0,A}\big(\sup_{\kappa}\mu^*_{r_0}(
h_{\kappa})\big)<\frac{\epsilon_{r_0}}{4}$, and
$F^3_{r_0,A}\big(\sup_{\kappa,\sigma}\mu^*_{r_0}(\tilde
h_{\kappa\iota})\big)<\frac{\epsilon_{r_0}}{4}$, whenever
$\|D^{r_0+1}u\|\leq\epsilon_{r_0}$. We may also assume that
$\|D^{r_0+1}u\|<\epsilon_{r_0}$ yields $u\in\V_{2}$. Then by (9.1)
$\|D^{r_0+1}u_h\|\leq\epsilon_{r_0}$, if
$\|D^{r_0+1}u\|\leq\epsilon_{r_0}$.

Next, we define $\epsilon_r$ for all $r>r_0$ inductively. Suppose
we have defined $\epsilon_{r_0}\ld \epsilon_{r-1}$.

In view of Propositions 4.6, 3.5, 5.7, 6.1 and 8.7, Lemma 3.6, and
the inequality (9.2) rewritten for $r$ with $6^r$ on the r.h.s.,
there exist  constants $\beta>0$ and $K=K_{\phi,\psi}$, and
polynomials $P_{\phi,\psi,r,A}$ without constant term such that
for all $u\in\V_{2}$  we have
\begin{equation}
\|D^{r+1}u_h\|\leq
A^{\beta-r}K^r\|D^{r+1}u\|+P_{\phi,\psi,r,A}(\sup_{s=0,1\ld
r}\|D^su\|).
\end{equation}
 Enlarging $A$ if necessary, suppose $A>K^{r_0}$.
Hence we have $A^{\beta-r}K^r<\frac{1}{4}$. Put
$b_r=P_{\phi,\psi,r,A}(\sup_{s=0,1\ld r}\|D^su\|)$, where
$\|D^{s+1}u\|\leq\epsilon_s$ for $s=r_0\ld r-1$.
 Then (9.3) can be rewritten as
\[\|D^{r+1}u_h\|\leq
\frac{1}{4}\|D^{r+1}u\|+b_r.\] Define $\epsilon_r=2b_r$. It
follows that $\|D^{r+1}u_h\|\leq\epsilon_r$ whenever
$\|D^{r+1}u\|\leq\epsilon_r$, as required.

\section{Proof of Corollary 1.2}

We have to check Epstein's axioms [4] for some basis of open sets
$\U$ of $M$ and $G=\con$:
\begin{enumerate}
\item If $U\in\U$ and $g\in G$ then $g(U)\in\U$. \item $G$ acts
transitively on $\U$. \item Let $g\in G$, $U\in\U$ and let
$\V\s\U$ be a covering of $M$. Then there are $s\geq 1$, $g_1\ld
g_s\in G$ and $V_1\ld V_s\in\V$ such that $g=g_1\ldots g_s$,
$\supp(g_i)\s V_i$ and $\supp(g_i)\cup g_{i-1}\ldots g_1(\overline
U)\neq M$ for $i=1\ld s$.
\end{enumerate}
In fact, let $U$ be any open ball in $M$ and $\U=\{g(U):g\in G\}$.
By using $\chi_A$, $\tau_{i,t}$, $i=0\ld n$, and $\sigma_{i,t}$,
$i=1\ld n$, see section 2, it is easily seen that $\U$ is a basis
and (2) is fulfilled. In view of Lemma 5.2 a standard reasoning
shows (3). Thus, due to [4] and Theorem 1.1, $\con$ is simple.

\section{Final remarks}
Let $G$ be a group and let $g\in [G,G]$. The commutator length
$cl_G(g)$ of $g$ is 0 if $g=e$, and is the least positive integer
$N$ such that $g=[g_1,h_1]\cdots[g_N,h_N]$ for some $g_i,h_i\in
G$, $i=1\ld N$, otherwise. Then $cl_G$ is a conjugation-invariant
norm on $[G,G]$, c.f. [3]. In the paper [3] by Burago, Ivanov and
Polterovich and in certain references therein a description of a
role played by conjugation-invariant norms on groups of geometric
origin is given.

As a trivial consequence of Theorem 1.1 we have
\begin{cor}
The commutator length is a conjugation-invariant norm on $\con$.
\end{cor}
It is known from several recent papers that the theorem of Banyaga
[1]  plays a clue role in the symplectic topology and geometry in
the sense that some invariants are expressed in terms of the
commutator length of related groups. It seems that, thanks to
Corollary 11.1, a similar role could be played by $cl_{\con}$ in
the contact topology and geometry.

Recall that a group is said to be {\it bounded} if it is bounded
w.r.t. any bi-invariant metric on it or, equivalently, any
conjugation-invariant norm on it is bounded. Recently, the problem
of boundedness was solved in many cases of $\diff_c(M)_0$, and the
solutions depend on the topology of $M$ (c.f. [3], [20]). In view
of Corollary 11.1 it is interesting to know whether $\con$ is
bounded and how it depends on $M$.

Another possible applications are related to Haefliger's
classifying spaces of contact foliations. Let
$B\overline{\cont}_c(M,\a)$ be the classifying space for the
foliated $C^{\infty}$ products with compact support with
transverse contact form. It is well known that
$B\overline{\cont}_c(M,\a)$ is the homotopy fiber of the mapping
\[B\cont_c(M,\a)^{\delta}_0\r B\cont_c(M,\a)_0,\] where the
superscript $\delta$ denotes the discrete topology. By an argument
similar to the  proof of Theorem 1.1 we have the following
\begin{thm}
 $H_1(B\overline{\cont}_c(M,\a);\Z)=H_1(\widetilde{\cont_c(M,\a)})=0$, where tilde
 indicates the universal covering group.
\end{thm}
 For the proof, see appendix.

 Up to my knowledge no version of the Thurston-Mather isomorphism
 (c.f. [15], [13], [2], [18], [19])
is known for $\con$. It seems likely that such a version could be
established, but a possible proof seems to be hard. This would
give information on the connectedness of  Haefliger's classifying
space for contact foliations.

 In [18] and [19] Tsuboi discussed the problem of the
connectedness of the Haefliger classifying spaces. It is  likely
that Theorem 1.1 is still true for the group $\cont^r_c(M,\a)$ of
contactomorphisms of class $C^r$ with $r$ large.

Observe that Theorem 11.2 reveals further fundamental difference
between the symplectic and the contact geometries. As it was
mentioned in the introduction the flux homomorphism plays a
crucial role in the geometry of  symplectic forms ([1], [2], [14])
(as well as in case of regular Poisson manifolds, c.f. [15], and
of locally conformal symplectic manifolds, c.f. [7]). The  domain
of the flux is the universal covering group of the group in
question. In view of Theorem 11.2 a possible analog of such a
homomorphism is necessarily trivial in the contact case.

\bigskip

\centerline{APPENDIX: THE PROOF OF THEOREM 11.2}

\medskip

Since the first equality is well-known it suffices to show the
second.

Let $G$ be a topological group. Denote by $\mathcal PG$ the
totality of paths  $\gamma:I\r G$ with $\gamma(0)=e$, where
$I=[0,1]$. The path group $\mathcal PG$ is a topological group
with the compact-open topology. Likewise, for a locally convex
vector space $V$ let $\mathcal PV$ be the totality of paths
$\gamma:I\r V$ with $\gamma(0)=0$. Then $\mathcal PV$ is a locally
convex vector space. If $X\s G$ (resp. $Y\s V$) are subsets
containing $e$ (resp. 0) then the subsets $\mathcal PX\s \mathcal
PG$ (resp. $\mathcal PY\s \mathcal PV$) are defined in the obvious
way.

Next, the symbol $\mathcal P_0G$ (resp. $\mathcal P_0V$) will
stand for the totality of  $\{f_t\}_{t\in I}\in\mathcal PG$ (resp.
$\{f_t\}_{t\in I}\in\mathcal PV$) such that $f_t=e$ (resp.
$f_t=0$) for $0\leq t\leq\frac{1}{2}$. The elements of $\mathcal
P_0G$ and $\mathcal P_0V$ will be called \wyr{special} paths. Note
that the subsets $\mathcal P_0X\s \mathcal P_0G$ (resp. $\mathcal
P_0Y\s \mathcal P_0V$) are well-defined for subsets $X\s G$ (resp.
$Y\s V$) with $e\in X$ (resp. $0\in Y$).

We have to show that $\widetilde{\cont_c(M,\a)}=\mathcal P\con
/\sim$ is a perfect group. Here $\sim$ denotes the relation of the
homotopy rel. endpoints. It is clear that for every
$\big[\{g_t\}\big]_{\sim},
\big[\{h_t\}\big]_{\sim}\in\widetilde{\cont_c(M,\a)}$, the classes
of them in $H_1(\widetilde{\cont_c(M,\a)})$ are equal
 whenever
$[\{g_t\}]=[\{h_t\}]$ in $H_1(\mathcal P\cont_c(M,\a))$. Take
arbitrarily
$\big[\{h_t\}\big]_{\sim}\in\widetilde{\cont_c(M,\a)}$, where
$\{h_t\}\in\mathcal P\con$. In view of Lemma 5.2 we may and do
assume that $\{h_t\}\in\mathcal P\cont_{I_A}(\R^m,\as)_0$. Observe
that Lemma 5.2 is still valid for the group  $\mathcal P_0\con$
instead of $\mathcal P\con$ and this fact is also used in the
proof of Lemma 8.4(3) for special paths.

In order to show that $\big[\{h_t\}\big]_{\sim}$ belongs to the
commutator subgroup of $\widetilde{\corr}$ we introduce suitable
changes in the subsequent sections.

In section 2 we single out special elements of $\mathcal P\corr$
as follows (c.f. (1)-(5) in section 2).  Abusing the notation they
will be designated as before. Namely,
$\tau_{i,t}=\{(\tau_{i,t})_s\}_{s\in I}$,
$\sigma_{i,t}=\{(\sigma_{i,t})_s\}_{s\in I}$,
$\chi_a=\{(\chi_a)_s\}_{s\in I}$, $\eta_a=\{(\eta_a)_s\}_{s\in I}$
are fixed elements of $\mathcal P\corr$ such that
$(\tau_{i,t})_s=\tau_{i,t}$, $(\sigma_{i,t})_s=\sigma_{i,t}$,
$(\chi_a)_s=\chi_a$ and $(\eta_a)_s=\eta_a$ for all
$\frac{1}{2}\leq s\leq 1$.

In section 4 the chart $
\Phi_{A}:\cont_{E}(\C_k^{m},\as)\supset\U_1\ni f\mapsto
u_f\in\V_2\s
 \CC^{\infty}_{E}(\C_k^m)$ induces the homeomorphism
 \[ \mathcal P\Phi_{A}:\mathcal P\cont_{E}(\C_k^{m},\as)
 \supset\mathcal P\U_1\ni \{f_t\}\mapsto\{ u_{f_t}\}\in\mathcal P\V_2\s
 \mathcal P\CC^{\infty}_{E}(\C_k^m).\]
 Notice that $\mathcal P\Phi_A$ preserves the subspaces of special
 paths.
We may and do assume that $\U_1^{-1}\cdot\U_1$ is contained in a
contractible neighborhood of the identity.

In section 5 by making use of $\mathcal P\Phi_{A}$ we define
$\{f_t\}^{\psi}$ for $\{f_t\}\in\mathcal P\U_1$ by putting
$\{f_t\}^{\psi}=\{f_t^{\psi}\}$.  Observe that
$\{f_t\}^{\psi}\in\mathcal P_0\U_1$ whenever $\{f_t\}\in\mathcal
P_0\U_1$, and  Propositions 5.4 holds. Proposition 5.7 holds for
isotopies in the sense that there is a decomposition for isotopies
and the estimates (1), (2) are satisfied for the corresponding
members of isotopies with the same constants and polynomials.
Next, Proposition 6.1 and the inclusion $(6.2)$ are still valid
for $\mathcal P_0\cont_{I_A}(\R^m,\as)_0$ in view of our new
definition of $\tau_{i,t}$, $\sigma_{i,t}$, $\chi_a$ and $\eta_a$
(with an analogous remark as for 5.7). Also for any
$\{f_t\}\in\mathcal P_0\cont_{
E^{(k+1)}_A}(\C^m_{k+1},\ac)_0^{(k)}$ there is $\{\hat
f_t\}\in\mathcal P_0\cont_{
E^{(k+1)}_A}(\C^m_{k+1},\ac)_0^{(k+1)}$ as in section 7.

In section 8 we have the operators $\mathcal P\Theta^{(k)}$ and
$\mathcal P\Xi^{(k)}$ on the relevant spaces of paths induced by
$\Theta^{(k)}$ and $\Xi^{(k)}$, resp. It is important that these
operators descend to the operators $\mathcal P_0\Theta^{(k)}$ and
$\mathcal P_0\Xi^{(k)}$ on the corresponding spaces of special
paths.  Lemmas 8.3, 8.4 and 8.6 remain valid on the spaces of
special paths and their proofs are completely analogous. All these
prerequisites lead to the rolling-up operator
\[\mathcal P_0\Psi_{A}:\mathcal P_0\cont_{J_A}(\R^m,\as)_0\cap \mathcal P_0\U_1
\r\mathcal P_0\cont_{K_{A}}(\R^m,\as)_0,\] which satisfies an
analogue of Proposition 8.7 (with a similar remark as the above
for 5.7). In particular, for any $\{g_t\}\in\dom(\mathcal
P_0\Psi_A)$ one has $[\mathcal P_0\Psi_{A}(\{g_t\})]=[\{g_t\}]$ in
$H_1(\mathcal P\corr)$.

In the proof of Theorem 11.2 we will use spaces of special paths
and the proof is completely analogous. Fix $A$, $r_0$ and
$\epsilon_r$ for $r\geq r_0$ as in section 9. Suppose that
$\mathcal L$ is  as in section 9. Then  $\mathcal P_0\L$ is a
convex  subset of the locally convex vector space $\mathcal
P_0\cc_{I_A}(\R^m)$. We may and do assume that $\sup_{t\in
I}\mu^*_{r_0}(\{h_t\})$ is sufficiently small since $\mathcal
P\cont_{I_A}(\R^m,\as)_0$ is a topological group. Moreover, there
is $\{\hat h_t\}\in\mathcal P_0\cont_{I_A}(\R^m,\as)_0$ such that
$\sup_{t\in I}\mu^*_{r_0}(\{\hat h_t\})$ is also sufficiently
small and $\big[\{\hat h_t\}\big]_{\sim}=\big[\{
h_t\}\big]_{\sim}$.

We define $\mathcal P_0\vartheta:\mathcal P_0\L\r\mathcal P_0\L$
by the formula $\mathcal
P_0\vartheta(\{u_t\})=\{\vartheta_t(u_t)\}$, where
$\vartheta_t:\L\r\L$ is determined by $\hat h_t$. Then there
exists $\{f_t\}\in(\mathcal P\Phi_A)^{-1}\mathcal P_0\L$ such that
$u_{f_1}=\Phi_A(f_1)$ is a fixed point of $\vartheta_1$, and
there is $\{g_t\}$ in the commutator subgroup of $\mathcal P\corr$
such that
$$\{k_t\}:=(\mathcal P\Phi_A)^{-1}\mathcal P_0\vartheta\mathcal P\Phi_A(\{f_t\})=\{f_t\}\cdot\{\hat h_t\}\cdot\{g_t\}$$
is an isotopy in $\U_1$. Since
$\Phi_A^{-1}\vartheta_1\Phi_A(f_1)=f_1$, it follows that
$\{f_t\}^{-1}\cdot\{k_t\}$ is a contractible loop. Therefore,
$\big[\{\hat h_t\}\big]_{\sim}=\big[\{g_t\}^{-1}\big]_{\sim}$ so
that the class of $\big[\{\hat h_t\}\big]_{\sim}$ is equal to $e$
in $H_1(\widetilde{\corr})$, as claimed.

\end{document}